\documentclass[reqno,11pt,a4paper]{amsart}
\usepackage{amsfonts,graphics,amssymb,amsmath,mathrsfs,graphicx,amsthm,sidecap}
\usepackage[mathscr]{euscript}
\usepackage{times}
\usepackage[usenames,dvipsnames]{xcolor}
\DeclareMathOperator{\id}{id}
\DeclareMathOperator{\spann}{span}		
\newcommand{\I}{{\mathbb I}}
\newcommand{\N}{{\mathbb N}}
\newcommand{\R}{{\mathbb R}}
\newcommand{\Z}{{\mathbb Z}}
\newcommand{\cA}{{\mathcal A}}
\newcommand{\cB}{{\mathcal B}}

\newcommand{\cU}{{\mathcal U}}
\newcommand{\eF}{{\mathscr F}}

\newcommand{\eps}{\varepsilon}
\newcommand{\tm}{\times}
\newcommand{\vphi}{\varphi}
\newcommand{\cAn}{\cA_n}
\newcommand{\phin}{\phi^n}
\newcommand{\intoo}[1]{\left(#1\right)}							
\newcommand{\set}[1]{\left\{#1\right\}}							
\newcommand{\abs}[1]{\left|#1\right|}							
\newcommand{\norm}[1]{\left\|#1\right\|}						
\newcommand{\dist}[1]{\operatorname{dist}\intoo{#1}}			

\newcommand{\fall}{\quad\text{for all }}

\newcommand{\eref}[1]{Ex.~\ref{#1}}
\newcommand{\lref}[1]{Lemma~\ref{#1}}
\newcommand{\pref}[1]{Prop.~\ref{#1}}
\newcommand{\tref}[1]{Thm.~\ref{#1}}
\newcommand{\cref}[1]{Cor.~\ref{#1}}
\newcommand{\fref}[1]{Fig.~\ref{#1}}

\newcommand{\sref}[1]{Sect.~\ref{#1}}
\newcommand{\Cd}{C_d}

\theoremstyle{plain}
\newtheorem{thm}{Theorem}[section]
\newtheorem{cor}[thm]{Corollary}
\newtheorem{lem}[thm]{Lemma}
\newtheorem{prop}[thm]{Proposition}
\newtheorem*{hyp}{Hypothesis}

\theoremstyle{definition}

\newtheorem{exam}[thm]{Example}
\newtheorem{rem}[thm]{Remark}
\numberwithin{equation}{section}
\allowdisplaybreaks
\begin{document}

\title[Attractors of integrodifference equations under discretization]{Numerical Dynamics of Integrodifference Equations: Forward Dynamics and Pullback Attractors\footnote{The work of Huy Huynh has been supported by the Austrian Science Fund (FWF) under grant number P 30874-N35.}\footnote{Peter E.\ Kloeden thanks the Universit\"at Klagenfurt for support and hospitality.}}
\author[Huy Huynh, Peter E.\ Kloeden, Christian P\"otzsche]{Huy Huynh$^\dag$, Peter E.\ Kloeden$^\ast$ and Christian P\"otzsche$^\dag$}
\address{\llap{$\dag$} Institut f\"ur Mathematik, 
Universit\"at Klagenfurt, 
9020 Klagenfurt, Austria}
\email{\{pham.huynh,christian.poetzsche\}@aau.at}
\address{\llap{$\ast$} Mathematisches Institut, 
Universit\"at T\"ubingen, 
72076 T\"ubingen, Germany}
\email{kloeden@na.uni-tuebingen.de}

\begin{abstract}
	In order to determine the dynamics of nonautonomous equations both their forward and pullback behavior need to be understood. For this reason we provide sufficient criteria for the existence of such attracting invariant sets in a general setting of nonautonomous difference equations in metric spaces. In addition it is shown that both forward and pullback attractors, as well as forward limit sets persist and that the latter two notions even converge under perturbation. As concrete application, we study integrodifference equation under spatial discretization of collocation type. 
\end{abstract}
\maketitle
%
%
%
\section{Introduction}
\subsection{Perturbation of attractors}
Obtaining a precise insight into the structure of a global attractor can be seen as the holy grail in the field of dissipative dynamical systems. This compact, invariant set attracting each bounded subset of the state space contains all equilibria, periodic solutions, homo- or heteroclinic connections and overall the essential dynamics. In general the structure of a global attractor can be rather complex and therefore it is no surprise that such objects behave sensitively w.r.t.\ perturbations and parameter variations \cite{hale:88}. However, under mild conditions global attractors are upper-semicontinuous in parameters. Criteria for lower semicontinuity and hence for full continuity in the Hausdorff metric are much harder to prove and often involve conditions being difficult to verify. Nevertheless it was shown in \cite{hoang:olson:robinson:15} that global attractors are at least continuous on a residual subset of the parameter space. 

For nonautonomous dynamical systems \cite{bortolan:carvalho:langa:20,carvalho:langa:robinson:10,kloeden:rasmussen:10,kloeden:yang:21} the range of possible long term behaviors, as well as the theory of attractors is understandably more involved. For instance, in order to obtain a full picture of the dynamics it might be necessary to consider different attractor concepts simultaneously. Although the notion of a pullback attractor shares several properties of the global attractor it is easy to construct systems with totally different forward dynamics but sharing the same pullback attractor \cite{kloeden:poetzsche:rasmussen:11}. This deficit stimulated the development of forward attractors \cite{kloeden:lorenz:16}. While a perturbation theory for pullback attractors can be found in \cite{bortolan:carvalho:langa:20,carvalho:langa:robinson:10} and their continuity on a residual set is due to \cite{hoang:olson:robinson:18}, related results for forward attractors are sparse, which might be due to their nonuniqueness and recent origin. An exception is \cite{kloeden:16} addressing the upper-semicontinuity of forward limit sets under time discretizations of finite-dimensional nonautonomous ODEs. 

Indeed, particularly important classes of perturbations originate in numerical analysis, when an equation generating a dynamical system is discretized. It actually is a key topic in Numerical Dynamics \cite{stuart:humphries:98} to determine how attractors behave under discretization. This research was initiated in \cite{kloeden:lorenz:86} for temporal discretizations of autonomous ODEs in $\R^d$, extensions to pullback attractors of nonautonomous ODEs can be found in \cite{cheban:kloeden:schmalfuss:00}, and related results for semiflows generated by PDEs are due to \cite{hale:lin:raugel:88,kloeden:lorenz:89} and \cite{stuart:95}. 

Forward attractors and limit sets of discrete time nonautonomous dynamical systems in infinite dimensions were constructed and studied in \cite{huynh:poetzsche:kloeden:20} providing a theoretical foundation. As a continuation, this paper addresses the persistence of forward attractors and establishes the upper-semicontinuity of forward limit sets under perturbations. Remaining in the setting of \cite{huynh:poetzsche:kloeden:20} our criteria are formulated for abstract difference equations in metric spaces. We believe that they apply to a wide range of evolutionary equations and their spatial discretizations. Nevertheless, although convergence under perturbation is proven, at this level of generality we are not able to establish a particular convergence rate without imposing further assumptions. 
\subsection{Integrodifference equations and spatial discretization}
As concrete application serve nonautonomous integrodifference equations (short IDEs)
\begin{equation}
	\tag{$I_0$}
	u_{t+1}=\int_\Omega f_t(\cdot,y,u_t(y))\,{\mathrm d} y
	\label{ide0}
\end{equation}
(see Sect.~\ref{sec3}). Such recursions in infinite-dimensional spaces originally stem from population genetics but gained popularity as models for the dispersal of species in theoretical ecology over the last decades \cite{kot:schaeffer:86,lutscher:19}. Recently also nonautonomous models were studied in \cite{bouhours:lewis:16,jacobsen:jin:lewis:15}, which clearly motivates our general approach. 

The several approaches to spatially discretize IDEs are based on related techniques for nonlinear integral equations \cite{atkinson:92}. Among them, in \sref{sec3} we focus on collocation methods approximating continuous functions on a compact habitat. As concrete collocation spaces we exemplify splines of different order. In this framework we provide conditions that pullback and forward attractors, as well as forward limit sets persist under collocation, and establish convergence of pullback attractors and forward limit sets for increasingly more accurate discretizations. Convergence rates could not be derived using our proof techniques, but we experimentally demonstrate in \sref{sec4} that quadratic convergence under piecewise linear collocation is preserved. We close with a construction of forward limit sets for asymptotically autonomous spatial Ricker equations. 
\subsection{Notation}
Let $\R_+:=[0,\infty)$. A \emph{discrete interval} $\I$ is the intersection of a real interval with the integers $\Z$, $\I':=\set{t\in\I:\,t+1\in\I}$ and in particular
\begin{align*}
	\I_\tau^+&:=\set{t\in\I:\,\tau\leq t},&
	\I_\tau^-&:=\set{t\in\I:\,t\leq\tau}\quad\text{for }\tau\in\Z,
\end{align*}
as well as $\N_0:=\set{0,1,2,\ldots}$. 

On a given metric space $(X,d)$, $\overline{A}$ denotes the closure of a subset $A\subseteq X$, $B_r(x):=\set{y\in X:\,d(x,y)<r}$ is the open ball with center $x\in X$ and radius $r>0$, and $\bar B_r(x)$ its closure. We write 
$
\mbox{dist}_A(x):=\inf_{a\in A}d(x,a)
$
for the \emph{distance} of a point $x\in X$ to a subset $A\subseteq X$ and
$
B_r(A):=\set{x\in X:\,\mbox{dist}_A(x)<r}
$
for its $r$-neighborhood. The \emph{Hausdorff semidistance} of bounded and closed subsets $A,B\subseteq X$ is defined as
$$
	\dist{A,B}:=\sup_{a\in A}\mbox{dist}_{B}(a),
$$
and, with a further closed and bounded subset $C\subseteq X$, one has the properties
\begin{align}
	\dist{A,B}\leq\dist{A,C}+\dist{C,B}. 
	\label{hs2}
\end{align}

A mapping $\eF:X\to X$ is called \emph{bounded}, if it maps bounded subsets of $X$ into bounded sets and \emph{globally bounded}, if $\eF(X)$ is bounded. A \emph{completely continuous} map $\eF$ is continuous and maps bounded sets into relatively compact sets.

Beyond this general terminology, we need to introduce terms commonly used in the area of nonautonomous dynamics \cite{kloeden:rasmussen:10}: 
A subset $\cA\subseteq\I\tm X$ with nonempty \emph{fibers} $\cA(t):=\{x\in X:\,(t,x)\in\cA\}$, $t\in\I$, is called \emph{nonautonomous set}. A nonautonomous set $\cA$ is said to have some topological property if each fiber $\cA(t)$, $t\in\I$, has this property. Furthermore, one speaks of a \emph{bounded} nonautonomous set $\cA$, if there exists real $R>0$ and a point $x_0\in X$ such that $\cA(t)\subseteq B_R(x_0)$ holds for all $t\in\I$. If $\I$ is unbounded above, then we define
$$
	\limsup_{t\to\infty}\cA(t):=\bigcap_{\tau\in\I}\overline{\bigcup_{\tau\leq t}\cA(t)}.
$$
\section{Nonautonomous difference equations and attractors}
\label{sec2}
Let $(X,d)$ be a complete metric space and $U_t\neq\emptyset$, $t\in\I'$, be closed subsets of $X$. We consider nonautonomous difference equations 
\begin{align} \label{deq}
	\tag{$\Delta$}
	\boxed{u_{t+1}=\eF_t(u_t)} 
\end{align}
having continuous right-hand sides $\eF_t:U_t\to X$, $t\in\I'$.

Given an \emph{initial time} $\tau\in\I$, a \emph{forward solution} $(\phi_t)_{\tau\leq t}$ to \eqref{deq} is a sequence satisfying $\phi_t\in U_t$ and $\phi_{t+1}=\eF_t(\phi_t)$ for all $\tau\leq t,t\in\I'$, whilst an \emph{entire solution} $(\phi_t)_{t\in\I}$ meets the same identity for all $t\in\I'$. The \emph{general solution} $\vphi(\cdot;\tau,u_\tau)$ of \eqref{deq} is the unique forward solution starting at $\tau\in\I$ in the \emph{initial state} $u_\tau\in U_\tau$, i.e.\ 
\begin{align} 
	\vphi(t;\tau,u_\tau):=
	\begin{cases}
		\eF_{t-1}\circ\ldots\circ\eF_\tau(u_\tau),&\tau<t,\\
		u_\tau,&\tau=t,
	\end{cases}
	\label{process}
\end{align}
as long as the iterates stay in $U_t$. 

From now on, assume that $\eF_t(U_t)\subseteq U_{t+1}$ holds for all $t\in\I'$ and abbreviate 
$$
	\cU:=\set{(t,u)\in\I\tm X:\,u\in U_t}.
$$
Then $\vphi$ satisfies the \emph{process} (or \emph{two-parameter semi-group}) \emph{property}
$$
	\vphi(t;s,\vphi(s;\tau,u))=\vphi(t;\tau,u)\fall\tau\leq s\leq t,\,u\in U_\tau.
$$

A nonautonomous set $\cA\subseteq\cU$ is said to be \emph{positively} or \emph{negatively invariant}, if
\begin{align*}
\eF_t\intoo{\cA(t)}&\subseteq\cA(t+1),&
\cA(t+1)&\subseteq\eF_t\intoo{\cA(t)}\fall t\in\I'
\end{align*}
holds, respectively. Accordingly, an \emph{invariant} set $\cA$ is both positively and negatively invariant, that is $\eF_t\intoo{\cA(t)}=\cA(t+1)$ for all $t\in\I'$.

A bounded nonautonomous set $\cA\subseteq\cU$ is called \emph{absorbing}, if for each $\tau\in\I$ and bounded $\cB\subseteq\cU$, there is an \emph{absorption time} $T=T(\tau,\cB)\in\N$ such that 
$$
	\vphi(t;\tau,\cB(\tau))\subseteq\cA(t)\fall t,\tau\in\I,\,t-\tau\geq T.
$$
In addition, for \emph{pullback} and \emph{forward absorbing sets}, the present time $t$ and the initial time $\tau$ work similarly to those in the properties of pullback and forward attraction (one notion of time is fixed and the other tends to infinity or minus infinity, cf.~\cite[pp.~44--45, Def.~3.3]{kloeden:rasmussen:10}). 
\subsection{Pullback attractors}
Let a discrete interval $\I$ be unbounded below. A \emph{pullback attractor} $\cA^\ast\subseteq\cU$ of \eqref{deq} is a compact and invariant nonautonomous set, which pullback attracts bounded sets $\cB\subseteq\cU$, i.e.\
\begin{align*}
	\lim_{s\to\infty} \dist{\vphi(\tau;\tau-s,\cB(\tau-s)),\cA^\ast(\tau)}=0 \fall\tau\in\I.
\end{align*}
With comprehensive information on pullback attractors given in \cite{bortolan:carvalho:langa:20,carvalho:langa:robinson:10,kloeden:rasmussen:10,poetzsche:10}, our existence condition for pullback attractors is based on the following notion: A difference equation \eqref{deq} is called \emph{pullback asymptotically compact}, if for every $\tau\in\I$, every sequence $(s_k)_{k\in\N}$ in $\N$ with $s_k\to\infty$ as $k\to\infty$ and every bounded sequence $(a_k)_{k\in\N}$ with $a_k\in\cU(\tau-s_k)$, the sequence $(\vphi(\tau;\tau-s_k,a_k))_{k\in\N}$ in $\cU(\tau)$ has a convergent subsequence. 
\begin{thm}[Existence of pullback attractors, {cf.~\cite[p.~19, Thm.~1.3.9]{poetzsche:10}}]
	\label{thmpbatt}
	Suppose a difference equation \eqref{deq} has a pullback absorbing set $\cA$. If \eqref{deq} is pullback asymptotically compact, then it possesses a unique pullback attractor $\cA^\ast\subseteq\cA$. 
\end{thm}
The importance of pullback attractors is due to their following dynamical characterization (see \cite[p.~17, Cor.~1.3.4]{poetzsche:10})
\begin{equation} 
	\cA^\ast
	=
	\set{(t,u)\in\cU:
	\begin{array}{l}
		\text{there exists a bounded entire}\\
		\text{solution $\phi$ to \eqref{deq} with }\phi_\tau=u
	\end{array}}
	\label{besol}
\end{equation}
motivating that pullback attractors are a suitable nonautonomous extension of global attractors \cite[p.~20, Thm.~1.9]{bortolan:carvalho:langa:20} of autonomous difference equations. 
\subsection{Forward dynamics}
On a discrete interval $\I$ unbounded above, we assume the right-hand sides $\eF_t:U\to U$, $t\in\I$, of \eqref{deq} act on and map into a common closed subset $U\neq\emptyset$ of $X$. A \emph{forward attractor} $\cA^+\subseteq\cU$ of \eqref{deq} is a compact and invariant nonautonomous set, which forward attracts bounded sets $\cB\subseteq\cU$, i.e.\
\begin{align*}
	\lim_{s\to\infty} \dist{\vphi(\tau+s;\tau,\cB(\tau)),\cA^+(\tau+s)}=0 \fall\tau\in\I.
\end{align*}
In contrast to pullback attractors $\cA^\ast$, forward attractors $\cA^+$ need not to be unique (cf.\ \cite[Sect.~4]{kloeden:yang:16}) and further properties are given in \cite{cui:kloeden:yang:19,huynh:poetzsche:kloeden:20,kloeden:lorenz:16,kloeden:rasmussen:10,kloeden:yang:16}. Not surprisingly, the existence of forward attractors is based on suitable compactness properties. For nonautonomous sets $\cA\subseteq\cU$, a difference equation \eqref{deq} is called
\begin{itemize}
	\item \emph{$\cA$-asymptotically compact}, if there exists a nonempty, compact set $K\subseteq U$ such that $K$ forward attracts $\cA(\tau)$, i.e.,\
	\begin{align*}
		\lim_{s\to\infty}\dist{\vphi(\tau+s;\tau,\cA(\tau)),K}=0 \fall\tau\in\I, 
	\end{align*}

	\item \emph{strongly $\cA$-asymptotically compact}, if there exists a nonempty, compact set $K\subseteq U$ such that every sequence $\intoo{(s_k,\tau_k)}_{k\in\N}$ in $\N\tm\I$ with $s_k\to\infty$, $\tau_k\to\infty$ as $k\to\infty$ yields
	$$
		\lim_{k\to\infty} \dist{\vphi(\tau_k+s_k;\tau_k,\cA(\tau_k)), K}=0. 
	$$
\end{itemize}
For positively invariant sets $\cA$, strong $\cA$-asymptotic compactness implies $\cA$-asymp\-totic compactness (cf.~\cite[Rem.~4.1]{huynh:poetzsche:kloeden:20}). 

In preparation of our next results, let us introduce a set $\Omega_{\cA}\subseteq\cU$ as
$$
	\Omega_{\cA}(\tau):=\bigcap_{0\leq s}\overline{\bigcup_{s\leq t}\vphi(\tau+t;\tau,\cA(\tau))}
	\fall\tau\in\I
$$
and the \emph{forward limit set} for the dynamics of \eqref{deq} starting from $\Omega_{\cA}$ is 
$$
	\omega_{\cA}^+:=\overline{\bigcup_{\tau\in\I}\Omega_{\cA}(\tau)}\subseteq U.
$$
\begin{thm}[Forward limit sets, {cf.~\cite[Thm.~4.6]{huynh:poetzsche:kloeden:20}}]
	\label{thmfwlim}
	Suppose a difference equation \eqref{deq} has a bounded and positively invariant set $\cA$. If \eqref{deq} is $\cA$-asymptotically compact with a compact $K\subseteq U$, then its forward limit set $\omega_\cA^+\subseteq K$ is nonempty, compact and forward attracts $\cA$. 
\end{thm}

In comparison to \tref{thmpbatt} a result guaranteeing the existence of forward attractors is more involved. It requires the assumption $\I=\Z$ and to introduce a further set $\cA^\star\subseteq\cU$ by its fibers
$$
	\cA^\star(\tau):=\bigcap_{0\leq s}\vphi(\tau;\tau-s,\cA(\tau-s))\fall\tau\in\Z.
$$
\begin{thm}[Forward attractors] 
	\label{thmfwatt}
	Let $\I=\Z$. 
	Suppose a difference equation \eqref{deq} has a closed, forward absorbing and positively invariant set $\cA$. If \eqref{deq} is pullback asymptotically compact, strongly $\cA$-asymp\-to\-tically compact with compact $K\subseteq U$ and
	\begin{itemize}
		\item[(i)] for every sequence $\intoo{(s_k,\tau_k)}_{k\in\N}$ in $\N\tm\Z$ with $\tau_k\to\infty$ as $k\to\infty$ one has
		\begin{align*}
			\lim_{k\to\infty}\dist{\vphi(\tau_k+s_k;\tau_k,K),K}=0,
		\end{align*}

		\item[(ii)] for all $\eps>0$ and $S\in\N$, there exists a $\delta>0$ such that for all $\tau\in\Z$ one has the implication
		\begin{align*}
			\left.
			\begin{array}{c}
				u,v\in\cA(\tau)\cup K,\\
				d(u,v)<\delta
			\end{array}
			\right\}
			\Rightarrow
			\max_{0<s\leq S}d(\vphi(\tau+s;\tau,u),\vphi(\tau+s;\tau,v))<\eps,
		\end{align*}

		\item[(iii)] $\omega_{\cA}^+=\limsup_{t\to\infty}\cA^\star(t)$, i.e.\ the forward limit sets for the dynamics starting from $\Omega_\cA$ and from within $\cA^\star$ coincide
\end{itemize}
	hold, then $\cA^\star\subseteq\cA$ is a forward attractor of \eqref{deq}.
\end{thm}
\begin{rem}
	(1) A situation, where an assumption such as (i) trivially holds, is when the compact set $K\subseteq U$ is positively invariant w.r.t.\ \eqref{deq}. 

	(2) The assumption (ii) can be interpreted as uniform continuity of every mapping $\vphi(\tau+s;\tau,\cdot)$ on $\cA(\tau)\cup K$ and finite discrete intervals $[\tau,\tau+s]\cap\Z$, $\tau\in\I$, $0<s\leq S$. Hence, if we assume that there exists a bounded $B\subseteq U$ satisfying
	$$
		\bigcup_{s\in\N_0}\intoo{\cA(\tau+s)\cup\vphi(\tau+s;\tau,K)}\subseteq B\fall\tau\in\Z, 
	$$
	and a Lipschitz condition with uniform constant $\ell\geq 0$ for each $\eF_t:U\to U$, $t\in\I$, on $B$, then
	$d(\vphi(\tau+s;\tau,u),\vphi(\tau+s;\tau,v))\leq\ell^sd(u,v)$ for all $u,v\in B$. From this we obtain that (ii) holds true. 
\end{rem}
\begin{proof}
	Invariance, being nonempty and compactness of $\cA^\star\subseteq\cA$ are established in \cite[Prop.~3.1]{huynh:poetzsche:kloeden:20}, while the forward attraction is due to \cite[Cor.~4.19]{huynh:poetzsche:kloeden:20}.
\end{proof}
\subsection{Perturbation of attractors}
Rather than sticking to a single problem \eqref{deq}, we now proceed to families of nonautonomous difference equations
\begin{align}
	\boxed{u_{t+1}=\eF_t^n(u_t)}
	\tag{$\Delta_n$}
	\label{deqn}
\end{align}
having continuous right-hand sides $\eF_t^n:U_t\to X$, $t\in\I'$, which are parameterized by $n\in\N_0$. The general solution of \eqref{deqn} is denoted by $\vphi^n$.

For $n\in\N$ the difference equations \eqref{deqn} may describe perturbations (concretely, numerical discretizations with accuracy increasing in $n$) of an initial problem $(\Delta_0)$. Hence, a crucial concept allowing us to compare the behavior of solutions to \eqref{deqn}, $n\in\N$, to the solutions of $(\Delta_0)$ is the \emph{local error}
$$
	e_t^n(u):=d(\eF_t^n(u),\eF_t(u)).
$$
We say the difference equations \eqref{deqn}, $n\in\N$, are \emph{convergent}, if there exists a \emph{convergence function} $\Gamma:\R_+\to\R_+$ satisfying $\lim_{\rho\searrow 0}\Gamma(\rho)=0$ such that for every bounded $B\subseteq U_t$, there exists a $C(B)\geq 0$ with
\begin{align}
	\sup_{u\in B}e_t^n(u)\leq C(B)\Gamma(\tfrac{1}{n})\fall t\in\I',\,n\in\N.
	\label{errest}
\end{align}
One speaks of a \emph{convergence rate} $\mu>0$, if $\Gamma(\rho)=\rho^\mu$. 

The following result determines how the local error evolves over time. It is crucial in order to apply our perturbation results in \tref{thmuscpbatt} and \ref{thmusfw}: 
\begin{prop}[Global discretization error] \label{properr}
	Let $\tau,T\in\I$ with $\tau\leq T$ be fixed, $u\in U_\tau$ and $\cB\subseteq\cU$ be a bounded nonautonomous set such that
	\begin{align} \label{properr1}
		\vphi^0(t;\tau,u)\in\cB(t)\fall\tau\leq t\leq T.
	\end{align}
	If the difference equations \eqref{deqn}, $n\in\N$, are convergent and for every $t\in\I'$ and bounded $B\subseteq X$ there exists a real $\ell_t(B)\geq 0$ such that 
	\begin{align} \label{properror1}
		d(\eF_t^n(u),\eF_t^n(v))
		\leq
		\ell_t(B)d(u,v)\fall n\in\N,\,u,v\in U_t\cap B,
	\end{align}
	then for every $\rho>0$ there exists an $N_0\in\N$ such that for all $n\geq N_0$ one has the inclusion $\vphi^n(t;\tau,u)\in B_\rho(\cB(t))$ and the global discretization error 
	\begin{align} \label{properr3}
		d(\vphi^n(t;\tau,u),\vphi^0(t;\tau,u))
		\leq
		\Gamma\left(\tfrac{1}{n}\right)\sum_{l_1=\tau}^{t-1}C(\cB(l_1))\prod_{l_2=l_1+1}^{t-1}\ell_{l_2}\bigl(B_\rho(\cB(l_2))\bigr)
	\end{align}
	for all $\tau\leq t\leq T$. 
\end{prop}
\begin{proof}
	Let $\rho>0$ and choose $(\tau,u)\in\cB$, $\tau\leq T$. By induction over $t\geq\tau$ we establish the existence of a $N_0=N_0(T)\in\N$ such that both \eqref{properr3} and
	\begin{equation}
		\vphi^n(t;\tau,u)\in B_\rho(\cB(t))\fall\tau\leq t\leq T,\,n\geq N_0
		\label{noind1}
	\end{equation}
	hold. This is trivially true for $t=\tau$. Assume that \eqref{properr3}--\eqref{noind1} are satisfied for some fixed $t<T$ and set $x_t:=d(\vphi^n(t;\tau,u),\vphi^0(t;\tau,u))$. Using the triangle inequality, 
	\begin{eqnarray*}
		x_{t+1}
		& \stackrel{\eqref{process}}{\leq} &
		d\bigl(\eF_t^n(\underbrace{\vphi^n(t;\tau,u)} _{\stackrel{\eqref{noind1}}{\in}B_\rho(\cB(t))}),\eF_t^n(\underbrace{\vphi^0(t;\tau,u)} _{\stackrel{\eqref{properr1}}{\in}\cB(t)})\bigr)
		+
		e_t^n(\vphi^0(t;\tau,u))\\		
%
		& \stackrel{\eqref{properror1}}{\leq} &
		\ell_t\bigl(B_\rho(\cB(t))\bigr)x_t
		+
		e_t^n(\underbrace{\vphi^0(t;\tau,u)}_{\stackrel{\eqref{properr1}}{\in}\cB(t)})\\		
		& \stackrel{\eqref{errest}}{\leq} &
		\ell_t\bigl(B_\rho(\cB(t))\bigr)x_t + C(\cB(t))\Gamma(\tfrac{1}{n})\\
		& \stackrel{\eqref{properr3}}{\leq} &
		\ell_t\bigl(B_\rho(\cB(t))\bigr)
		\Gamma(\tfrac{1}{n})\sum_{l_1=\tau}^{t-1}C(\cB(l_1))\prod_{l_2=l_1+1}^{t-1}\ell_{l_2}\bigl(B_\rho(\cB(l_2))\bigr)
		+C(\cB(t))\Gamma(\tfrac{1}{n})\\
		& = &
		\Gamma(\tfrac{1}{n})\sum_{l_1=\tau}^tC(\cB(l_1))\prod_{l_2=l_1+1}^t\ell_{l_2}\bigl(B_\rho(\cB(l_2))\bigr)
	\end{eqnarray*}
	results. For a sufficiently large $N_0$ it holds $d(\vphi^n(t+1;\tau,u),\vphi^0(t+1;\tau,u))<\rho$ and the desired inclusion $\vphi^n(t+1;\tau,u)\in B_\rho(\cB(t+1))$ results. 
\end{proof}

Rather than providing suitable conditions on \eqref{deqn} guaranteeing that the existence of a pullback attractor $\cA_0^\ast$ for $(\Delta_0)$ implies that also the perturbations \eqref{deqn}, $n\in\N$, have pullback attractors $\cA_n^\ast$ (\emph{persistence}), on the present abstract level we simply assume that such attractors $\cA_n^\ast$, $n\in\N_0$, do exist. Then the next result is a discrete time version of \cite[p.~152, Thm.~7.13]{bortolan:carvalho:langa:20}.
\begin{thm}[Upper-semicontinuity of pullback attractors] \label{thmuscpbatt}
	Let $\I$ be unbounded below. If each difference equation \eqref{deqn}, $n\in\N_0$, has a pullback attractor $\cA_n^\ast$ and 
	\begin{itemize}
		\item[(i)] $\bigcup_{n\in\N_0}\cAn^\ast(\tau)$ is relatively compact for every $\tau\in\I$, 

		\item[(ii)] $\bigcup_{n\in\N_0}\bigcup_{s\leq\tau}\cAn^\ast(s)$ is bounded for every $\tau\in\I$, 

		\item[(iii)] (\emph{convergence condition}) for all $\eps>0$, $S\in\N$ and every compact $K\subseteq X$ there exists an $n_0\in\N$ such that for all $\tau\in\I$ one has the implication
		$$
			n>n_0
			\Rightarrow
			\max_{0<s\leq S}\sup_{u\in U_{\tau-s}\cap K}d(\vphi^n(\tau;\tau-s,u),\vphi^0(\tau;\tau-s,u))
			<
			\eps
		$$
	\end{itemize}
	hold, then $\lim_{n\to\infty}\max_{t\in I}\dist{\cAn^\ast(t),\cA_0^\ast(t)}=0$ for every bounded $I\subseteq\I$. 
\end{thm}
\begin{proof}
	Let $n\in\N_0$. Recall from \eqref{besol} that each fiber $\cAn^\ast(t)\subseteq\cU(t)$, $t\in\I$, consists of initial values for a bounded entire solutions $\phin$ to \eqref{deqn}. The proof is then divided into two claims. 

	(I) Claim: \emph{There is a bounded entire solution $\phi^0$ to $(\Delta_0)$ and a subsequence $\intoo{\phi^{n_k}}_{k\in\N}$ of $\intoo{\phin}_{n\in\N}$ such that $\phi^0_t\in\cA_0^\ast(t)$ for all $t\in\I$ and $\intoo{\phi^{n_k}}_{k\in\N}$ converges to $\phi^0$ uniformly on bounded subintervals of $\I$.}
		
	In fact, $\intoo{\phin_\tau}_{n\in\N}$ can be seen as sequence in $\overline{\bigcup_{n\in\N_0}\cAn^\ast(\tau)}$ with $\tau\in\I$. Combining this with the fact from (i) that $\overline{\bigcup_{n\in\N_0}\cAn^\ast(\tau)}$ is compact yields the existence of an infinite subset $N_0\subseteq\N$ so that the subsequence $\intoo{\phin_\tau}_{n\in N_0}$ of $\intoo{\phin_\tau}_{n\in\N}$ converges to a limit $u_\tau\in\overline{\bigcup_{n\in\N_0}\cAn^\ast(\tau)}$. If we set $\phi^0_t:=\vphi^0(t;\tau,u_\tau)$ for each $t\in\I_\tau^+$, then $\intoo{\phin}_{n\in N_0}$ converges to $\phi^0$ uniformly on bounded subintervals of $\I_\tau^+$ due to (iii).
	
	Now by mathematical induction, we suppose that there exists a bounded forward solution $(\phi_t^0)_{s\leq t}$, $s\leq\tau$, to $(\Delta_0)$ as well as infinite subsets $N_s\subset N_{s-1}$ such that $\intoo{\phin}_{n\in N_s}$ converges to $\phi^0$ uniformly in bounded subintervals of $\I_s^+$. Similarly to the above, since the sequence $\intoo{\phin_{s-1}}_{n\in N_s}$ is contained in the compact set $\overline{\bigcup_{n\in\N_0}\cAn^\ast(s-1)}$, there is an infinite subset $N_{s+1}\subset N_s$ so that the subsequence $\intoo{\phin_{s-1}}_{n\in N_{s+1}}$ of $\intoo{\phin_{s-1}}_{n\in N_s}$ is convergent to $u_{s-1}\in\overline{\bigcup_{n\in\N_0}\cAn^\ast(s-1)}$. Therefore, $(\phi_t^0)_{s-1\leq t}$ is a bounded solution to $(\Delta_0)$ and again by (iii), the sequence $\intoo{\phin}_{n\in N_{s+1}}$ converges to $\phi^0$ uniformly in bounded subintervals of $\I_{s-1}^+$.
	
	Repeating this procedure eventually results in a bounded entire solution $\phi^0$ to the equation $(\Delta_0)$ and thus, $\phi^0_t\in\cA_0^\ast(t)$ for each $t\in\I$ due to (ii). Moreover, each sequence $\intoo{\phin_t}_{n\in\N}$ in $\cAn^\ast(t)$ has a uniformly convergent subsequence $\intoo{\phi^{n_s}_t}_{n_s\in N_s}$ converging to the limit $\phi^0_t\in\cA_0^\ast(t)$ for each $t\in\I$ as $n_s$ is the $s$-th element of $N_s\subseteq\N$, $s\in\N$. This completes the proof of the first claim.
	
	(II) Claim: \emph{$\lim_{n\to\infty}\dist{\cAn^\ast(t),\cA_0^\ast(t)}=0$ for all $t\in\I$.}\\
	By means of contradiction, assume the limit in the claim fails for some instant $t\in\I$, i.e.\ there are $\eps_0>0$ and $n_0\in\N$ such that $\dist{\cA_{n_0}^\ast(t),\cA_0^\ast(t)}\geq\eps_0$ holds. Because $\cA_{n_0}^\ast(t)$ is compact, there exists a bounded entire solution $\phi^{n_0}$ with $\phi_t^{n_0}\in\cA_{n_0}^\ast(t)$ so that $\mbox{dist}_{\cA_0^\ast(t)}(\phi^{n_0}_t)\geq \eps_0$, a contradiction to claim (I) as for every bounded entire solution $\phin$, the values $\phi_t^n\in\cAn^\ast(t)$ have a convergent subsequence with limit in $\cA_0^\ast(t)$.
	
	Because bounded subsets $I\subseteq\I$ are finite, the assertion follows from (II). 
\end{proof}

In contrast, a related result addressing the forward dynamics inherently yields the existence of limit sets: 
\begin{thm}[Upper-semicontinuity of forward limit sets]
	\label{thmusfw}
	Let $\I$ be unbounded above. Suppose a bounded set $\cA\subseteq\cU$ is positively invariant for every difference equation \eqref{deqn}, $n\in\N_0$. If each \eqref{deqn}, $n\in\N_0$, is strongly $\cA$-asymp\-to\-ti\-cally compact with a compact $K_n\subseteq U$, then its forward limit set $\omega_{\cA,n}^+\subseteq K_n$ is nonempty, compact and forward attracts $\cA$. If $K:=\bigcup_{n\in\N_0}K_n$ and moreover 
	\begin{itemize}	
		\item[(i)] $\omega_{\cA,0}^+$ attracts $K$ uniformly in $\tau\in\I$, i.e.,
		$$
			\lim_{s\to\infty}\sup_{\tau\in\I}\dist{\vphi^0(\tau+s;\tau,K),\omega_{\cA,0}^+}=0, 
		$$

		\item[(ii)] for every sequence $\intoo{(s_k,\tau_k)}_{k\in\N}$ in $\N\tm\I$ with $\tau_k\to\infty$ as $k\to\infty$ one has
		$$
			\lim_{k\to\infty}\dist{\vphi^n(\tau_k+s_k;\tau_k,K_n),K_n}=0\fall n\in\N_0,
		$$
		
		\item[(iii)] for all $\eps>0$, $S\in\N$ and $n\in\N_0$, there exists a $\delta>0$ such that for all $\tau\in\I$ one has the implication
		$$
			\left.
			\begin{array}{c}
				u,v\in\cA(\tau)\cup K_n,\\
				d(u,v)<\delta
			\end{array}
			\right\}
			\Rightarrow
			\max_{0<s\leq S} d(\vphi^n(\tau+s;\tau,u),\vphi^n(\tau+s;\tau,v))<\eps,
		$$
		
		\item[(iv)] (\emph{convergence condition}) for all $\eps>0$, $S\in\N$, there exists an $n_0\in\N$ such that for all $\tau\in\I$ one has the implication
		$$
			\left.
			\begin{array}{c}
				u\in\cA(\tau)\cup K,\\
				n>n_0
			\end{array}
			\right\}
			\Rightarrow
			\max_{0<s\leq S} d(\vphi^n(\tau+s;\tau,u),\vphi^0(\tau+s;\tau,u))<\eps
		$$
	\end{itemize}
	hold, then $\lim_{n\to\infty}\dist{\omega_{\cA,n}^+,\omega_{\cA,0}^+}=0$.
\end{thm}
\begin{proof}
	To begin with, recall from \tref{thmfwlim} that each \eqref{deqn}, $n\in\N_0$, has a nonempty, compact, forward limit set $\omega_{\cA,n}^+\subseteq K_n$. Moreover, thanks to \cite[Thm.~4.10]{huynh:poetzsche:kloeden:20}, $\omega_{\cA,n}^+$ is \emph{asymptotically negatively invariant}, i.e., for all $u\in\omega_{\cA,n}^+$, $\eps>0$ and $T>0$, there is an integer $S_\eps$ satisfying $\tau+S_\eps-T\in\I$ and a point $u_\eps\in\omega_{\cA,n}^+$ such that
	$$
		d(\vphi^n(\tau+S_\eps;\tau+S_\eps-T,u_\eps),u)<\eps\fall\tau\in\I,\,n\in\N_0. 
	$$
	In order to establish $\lim_{n\to\infty}\dist{(\omega_{\cA,n}^+,\omega_{\cA,0}^+)}=0$ we proceed by contradiction. Assume that the above limit does not hold. Then there is a sequence $(n_j)_{j\in\N}$ in $\N$ satisfying $n_j\to\infty$ as $j\to\infty$ and an $\eps_0>0$ such that
	\begin{align*}
		4\eps_0\leq\dist{\omega_{\cA,n_j}^+,\omega_{\cA,0}^+}\fall j\in\N.
	\end{align*}
	By the compactness of $\omega_{\cA,n_j}^+$, the supremum in the definition of $\dist{\omega_{\cA,n_j}^+,\omega_{\cA,0}^+}$ is attained and so is the existence of a point $w_j\in\omega_{\cA,n_j}^+$ that satisfies
	\begin{align} \label{clmIathm26}
		4\eps_0
		\leq
		\dist{\omega_{\cA,n_j}^+,\omega_{\cA,0}^+}
		=
		\mbox{dist}_{\omega_{\cA,0}^+}(w_j).
	\end{align}
	On the other hand, from (i), there is a time $T_0\in\N$ such that
	\begin{align} \label{clmIbthm26}
		\dist{\vphi^0(\tau+s;\tau,K),\omega_{\cA,0}^+}<\eps_0
		\fall s\geq T_0,\,\tau\in\I.
	\end{align}
	Since $T_0>0$ and $w_j\in\omega_{\cA,n_j}^+$, the asymptotic negative invariance of $\omega_{\cA,n_j}^+$ yields that there is an $s_0=s_0({\eps_0})\in\I$ with $t_0:=\tau+s_0$, a $\tau_0:=t_0-T_0\in\I$ and a $u_j=u_j(n_j,T_0)\in\omega_{\cA,n_j}^+\subseteq K_{n_j}$ so that 
	$$
		d\left(w_j,\vphi^{n_j}(t_0;\tau_0,u_j)\right)<\eps_0.
	$$
	Now notice that as $u_j\in\cA(\tau_0)\cup K$, with integers $N=N(\eps_0,T_0)>0$ independent of $\tau_0$ and $n_j\in\Z^+_{N}$, our assumption (iv) yields
	\begin{align*}
		d(\vphi^{n_j}(t_0;\tau_0,u_j), \vphi^0(t_0;\tau_0,u_j))<\eps_0, 
	\end{align*}
	as well as \eqref{clmIbthm26} results in
	\begin{align*}
		\mbox{dist}_{\omega_{\cA,0}^+}(\vphi^0(t_0;\tau_0,u_j))
		<
		\dist{\vphi^0(t_0;\tau_0,K),\omega_{\cA,0}^+}<\eps_0.
	\end{align*}
	Combining the last three inequalities, we obtain
	\begin{multline*}
		\mbox{dist}_{\omega_{\cA,0}^+}(w_j)
		\leq
		d\left(w_j,\vphi^{n_j}(t_0;\tau_0,u_j)\right)
		+ 
		d\left(\vphi^{n_j}(t_0;\tau_0,u_j),\vphi^0(t_0;\tau_0,u_j)\right) \\
		+
		\mbox{dist}_{\omega_{\cA,0}^+}(\vphi^0(t_0;\tau_0,u_j))
		<
		\eps_0+\eps_0+\eps_0=3\eps_0,
	\end{multline*}
	a contradiction to \eqref{clmIathm26}. The proof is done. 
\end{proof}
\section{Integrodifference equations under discretizations}
\label{sec3}
In this Section we apply our abstract perturbation theory for general difference equations $(\Delta_0)$ to nonautonomous integrodifference equations. They are recursions of the form
\begin{equation}
	\tag{$I_0$}
	\boxed{u_{t+1}=\int_\Omega f_t(\cdot,y,u_t(y))\,{\mathrm d} y,}
\end{equation}
where $\Omega\subset\R^\kappa$ is assumed to be nonempty and compact. As state space for \eqref{ide0} we consider an ambient subset of the $\R^d$-valued continuous functions over $\Omega$, abbreviated by $\Cd$ and equipped with the maximum norm $\norm{u}:=\max_{x\in\Omega}\abs{u(x)}$. 
\subsection{Basics and collocation}
Throughout we suppose the following standing assumption in order to obtain well-definedness for \eqref{ide0}: 
\begin{hyp}
	Suppose for all $t\in\I'$ that a subset $Z_t\subseteq\R^d$ is nonempty and closed, while a kernel function $f_t:\Omega\tm\Omega\tm Z_t\to\R^d$ is continuous and satisfies
	\begin{itemize}
		\item[$(H_1)$] there exist $\alpha_t,\beta_t:\Omega^2\to\R_+$ measurable in the second argument with
		\begin{align*}
			a_t&:=\sup_{x\in\Omega}\int_\Omega \alpha_t(x,y)\,{\mathrm d} y<\infty,&
			b_t&:=\sup_{x\in\Omega}\int_\Omega \beta_t(x,y)\,{\mathrm d} y<\infty
		\end{align*}
		and
		\begin{equation}
			\abs{f_t(x,y,z)}
			\leq
			\beta_t(x,y)+\alpha_t(x,y)|z|
			\fall x,y\in\Omega,\,z\in Z_t, 
			\label{assumpK0}
		\end{equation}

		\item[$(H_2)$] $\int_\Omega f_t(x,y,u(y))\,{\mathrm d} y\in Z_{t+1}$ for all $x\in\Omega$ and $u\in U_t$ with
		$$
			U_t:=\set{u\in\Cd:\,u(x)\in Z_t\text{ for all }x\in\Omega}.
		$$
	\end{itemize}
\end{hyp}
The continuity of the kernel functions $f_t$ guarantees that the Urysohn operators
\begin{align}
	\eF_t:U_t&\to\Cd,&
	\eF_t(u)&:=\int_\Omega f_t(\cdot,y,u(y))\,{\mathrm d} y
	\label{rhs}
\end{align}
are completely continuous on the closed sets $U_t$ (cf.~\cite{martin:76,poetzsche:18}). Furthermore, Hypothesis $(H_2)$ ensures the inclusion $\eF_t(U_t)\subseteq U_{t+1}$ for all $t\in\I'$. In particular, \eqref{ide0} is a special case of $(\Delta_0)$ in the metric space $X=\Cd$, $d(u,v)=\norm{u-v}$. 

In order to simulate IDEs \eqref{ide0} on a computer, finite-dimen\-sional approximations of their right-hand sides \eqref{rhs} and of their state space $C_d$ are due. For this purpose, choose linearly independent functions $\phi_1,\ldots,\phi_{d_n}\in C_1$ yielding an \emph{ansatz space} $X_n:=\spann\set{\phi_1,\ldots,\phi_{d_n}}$ of finite dimension $d_n$. With a projector $\pi_n\in L(C_1)$ onto the space $X_n$ it is clear that 
\begin{align*}
	\Pi_n&\in L(\Cd),&
	\Pi_n(u_j)_{j=1}^d&:=(\pi_nu_j)_{j=1}^d
\end{align*}
is a projector onto the Cartesian product $X_n^d$. Supplementing this, for convenience we define $\Pi_0:=\id_{C_d}$. In the following, we impose the following stability assumption (cf.~\cite[p.~50, Def.~4.8]{hackbusch:14}): 
\begin{hyp}
	\begin{itemize}
		\item[$(H_3)$] The projections $\Pi_n$ are uniformly bounded, i.e.\ 
		\begin{equation}
			p:=\sup_{n\in\N}\norm{\Pi_n}\in[1,\infty),
			\label{condstab}
		\end{equation}

		\item[$(H_4)$] $\Pi_nU_t\subseteq U_t$ for all $n\in\N$ and $t\in\I'$. 
	\end{itemize}
\end{hyp}
A \emph{collocation method} is a spatial discretization of \eqref{ide0} of the form
\begin{align}
	\tag{$I_n$}
	u_{t+1}&=\eF_t^n(u_t),&
	\eF^n_t&:=\Pi_n\eF_t:U_t\to X_n^d. 
	\label{iden}
\end{align}
Thanks to $(H_4)$ these difference equations are well-defined and fit into the general framework of \eqref{deqn}, $n\in\N$. Based on Hypothesis $(H_1)$ we next identify absorbing sets for IDEs \eqref{ide0} and their spatial discretizations. 
\begin{prop}[Absorbing sets for \eqref{iden}]\label{propabc}
	Let $\rho>1$ and suppose $(H_1$--$H_4)$ hold.
	\begin{enumerate}
		\item If $\I$ is unbounded below, and for all $\tau\in\I$ one has
		\begin{align}
			\lim_{s\to\infty}\prod_{l=\tau-s}^{\tau-1}pa_l&=0,&
			R_\tau&:=p\sum_{l_1=-\infty}^{\tau-1}b_{l_1}\prod_{l_2=l_1+1}^{\tau-1}pa_{l_2}\in(0,\infty)
			\label{propabc1}
		\end{align}
		and $\sup_{\tau\in\I}R_\tau<\infty$, then $\cA:=\set{(\tau,u)\in\cU:\,\norm{u}\leq\rho R_\tau}$ is pullback absorbing for each col\-location discretization \eqref{iden}, $n\in\N_0$. 

		\item If $\I$ is unbounded above, and for all $\tau\in\I$ one has
		\begin{align}
			\lim_{s\to\infty}\prod_{l=\tau}^{\tau+s-1}pa_l&=0,&
			R_\tau&:=p\lim_{t\to\infty}\sum_{l_1=\tau}^{t-1}b_{l_1}\prod_{l_2=l_1+1}^{t-1}pa_{l_2}\in(0,\infty)
			\label{propabc2}
		\end{align}
		and $\sup_{\tau\in\I}R_\tau<\infty$, then $\cA:=\set{(\tau,u)\in\cU:\,\norm{u}\leq\rho R_\tau}$ is forward absorbing for each collocation discretization \eqref{iden}, $n\in\N_0$. 
	\end{enumerate}
\end{prop}
\begin{rem}
	(1) When dealing with the initial equation \eqref{ide0}, i.e.\ in case $n=0$, one can choose $p=1$. 

	(2) If the reals $a_t=:a$, $b_t=:b$ obtained from the linear growth assumption \eqref{assumpK0} are constant in time, then the condition $a\in[0,\tfrac{1}{p})$ implies that both the pullback and forward absorbing sets have constant fibers
	$$
		\cA=\set{(t,u)\in\cU:\,\norm{u}\leq\tfrac{\rho pb}{1-pa}}.
	$$
	In particular, for constant functions $\alpha_t(x,y)\equiv:\alpha$, $\beta_t(x,y)\equiv:\beta$ in \eqref{assumpK0}, then $a=\alpha\mu_\kappa(\Omega)$, $b=\beta\mu_\kappa(\Omega)$, 
	where $\mu_\kappa(\Omega)>0$ is the Lebesgue measure of $\Omega\subset\R^\kappa$.
\end{rem}
\begin{proof}
	Let $t\in\I'$, $u\in U_t$ and $n\in\N_0$. For every $x\in\Omega$ we obtain
	\begin{eqnarray*}
		\abs{\eF_t^n(u)(x)}
		& \stackrel{\eqref{condstab}}{\leq} &
		p\abs{\eF_t(u)(x)}
		\stackrel{\eqref{ide0}}{\leq}
		p\int_\Omega\abs{f_t(x,y,u(y))}\,{\mathrm d} y\\
		& \stackrel{\eqref{assumpK0}}{\leq} &
		p\int_\Omega \beta_t(x,y)\,{\mathrm d} y+p\int_\Omega\alpha_t(x,y)\abs{u(y)}\,{\mathrm d} y
		\leq
		pb_t+pa_t\norm{u}
	\end{eqnarray*}
	and passing to the supremum over $x\in\Omega$ yields $\norm{\eF_t(u)}\leq pb_t+pa_t\norm{u}$. Then both assertions (a) and (b) for \eqref{iden} result from \cite[Prop.~3.2 resp.\ 4.16]{huynh:poetzsche:kloeden:20}, where the condition $\sup_{\tau\in\I}R_\tau<\infty$ guarantees boundedness of $\cA$. \end{proof}

\begin{hyp}
	\begin{itemize}
		\item[$(H_5)$] For every $t\in\I'$ and $r>0$, there exists a function $\lambda_t:\Omega\tm\Omega\to\R_+$ measurable in the second argument with
		\begin{equation}
			\abs{f_t(x,y,z)-f_t(x,y,\bar z)} \leq \lambda_t(x,y)\abs{z-\bar z}
			\fall z,\bar z\in Z_t\cap\bar B_r(0)
			\label{lemlip1}
		\end{equation}
		and $x,y\in\Omega$, where $\ell_t(r):=\sup_{x\in\Omega}\int_\Omega\lambda_t(x,y)\,{\mathrm d} y<\infty$.
	\end{itemize}
\end{hyp}

\begin{lem}\label{lemlip}
	Suppose $(H_1$--$H_3)$ and $(H_5)$ hold. For each $t\in\I'$ and $r>0$ one has the Lipschitz estimate
	\begin{equation}
		\norm{\eF_t^n(u)-\eF_t^n(\bar{u})} \leq p\ell_t(r)\norm{u-\bar{u}}
		\fall n\in\N_0,\,u,\bar{u}\in U_t\cap\bar B_r(0).
		 \label{lemlip2}
	\end{equation}
\end{lem}
\begin{proof}
	Using \eqref{condstab} the claim results from \cite[Cor.~B.6]{poetzsche:18}. 
\end{proof}

The local discretization error in collocation discretizations \eqref{iden} depends on the smoothness of the image functions $\eF_t(u):\Omega\to\R^d$. Thereto, given a subspace $X(\Omega)\subset\Cd$ one denotes a right-hand side \eqref{rhs} as $X(\Omega)$-\emph{smoothing}, if $\eF_t(u)\in X(\Omega)$ for all $t\in\I'$ and $u\in U_t$. For instance, convolution kernels yield that $\eF_t(u)$ have a higher order smoothness than the $f_t$ itself (cf.~\cite[pp.~50ff, Sect.~3.4]{hackbusch:95}). 

Note that for collocation methods based on polynomial interpolation the stability condition \eqref{condstab} fails even with Chebyshev nodes as collocation points (see \cite[p.~52]{hackbusch:14}). Nevertheless, the stability condition \eqref{condstab} holds when working with spline functions as basis of $X_n$. For this purpose, we restrict to $\Omega=[a,b]$, choose a grid
\begin{equation}
	a=:x_0<x_1<\ldots<x_n:=b
	\label{grid}
\end{equation}
and abbreviate $\overline{h}_n:=\max_{j=0}^{n-1}(x_{j+1}-x_j)$, $\underline{h}_n:=\min_{j=0}^{n-1}(x_{j+1}-x_j)$. Beyond that assume that the grid \eqref{grid} is extended to a strictly increasing sequence $(x_j)_{j\in\Z}$. 

Let $\beta_j^l:\R\to\R_+$, $j\in\Z$, denote the \emph{$B$-splines} of degree $l\in\N$ introduced e.g.\ in \cite[pp.~242ff]{hammerlin:hoffmann:91} and consider the subspaces
$$
	X_n:=\spann\set{\beta_{-l}^l,\ldots,\beta_{n-1}^l}\subset C[a,b]. 
$$
\begin{exam}[linear splines]\label{exp1}
	The $B$-splines $\beta_j^1:\R\to[0,1]$, $-1\leq j\leq n$ (the hat functions), form a basis of the ansatz space consisting of piecewise linear functions $X_n\subset C[a,b]$ having dimension $d_n=n+1$. The spline projections 
	$$
		\pi_nu:=\sum_{j=0}^nu(x_j)\beta_{j-1}^1
	$$
	fulfill to the interpolation conditions $(\pi_nu)(x_j)=u(x_j)$, $0\leq j\leq n$, satisfy $\norm{\pi_n}=1$ (see \cite[p.~450]{atkinson:han:05}), hence $p=1$ in \eqref{condstab}, and the error estimate \cite[p.~275]{hammerlin:hoffmann:91}
	\begin{align*}
	\norm{u-\pi_nu}
	\leq
	\tfrac{\overline{h}_n^2}{8}\norm{u''}\fall n\in\N,\,u\in C^2[a,b].
	\end{align*}
	Consequently, if \eqref{ide0} is $C^2[a,b]^d$-smoothing, then an application to \eqref{iden} yields 
	\begin{align}
		e_t^n(u)
		\leq
		\tfrac{\overline{h}_n^2}{8}\norm{\eF_t(u)''}\fall t\in\I',\,u\in U_t
		\label{b1est}
	\end{align}
	for the local discretization error. The same quadratic convergence also holds, when multidimensional piecewise linear interpolation is used on domains $\Omega\subset\R^\kappa$ (see \cite[Sect.~3.1.3]{poetzsche:18}). 
\end{exam}
\begin{exam}[quadratic splines]
	Given a grid \eqref{grid} we introduce the $n+2$ collocation points
	\begin{align*}
		\xi_0&:=a,&
		\xi_j&:=\tfrac{x_j+x_{j-1}}{2}\fall 0\leq j\leq n,&
		\xi_{n+1}=:&b.
	\end{align*}
	Then the quadratic splines $\beta_j^2:\R\to[0,1]$, $-2\leq j\leq n$, yield an ansatz space $X_n\subset C^1[a,b]$ of dimension $d_n=n+2$, where the interpolation conditions $(\pi_nu)(\xi_j)=u(\xi_j)$, $0\leq j\leq n+1$, result in a projection $\pi_n:C[a,b]\to X_n$ satisfying $\norm{\pi_n}\leq 2$ (cf.\ \cite[Cor.~3.2]{kammerer:etal:74}), thus $p=2$ in \eqref{condstab} and 
	$$
		\norm{u-\pi_nu}
		\leq
		\tfrac{\overline{h}_n^3}{24}\bigl\|u^{(3)}\bigr\|
		\fall n\in\N,\,u\in C^3[a,b]
	$$
	as interpolation error (see \cite[Thm.~6]{dubeau:savoie:96}). Hence, for $C^3[a,b]^d$-smoothing equations \eqref{ide0} the local discretization error becomes
	\begin{align*}
	e_t^n(u)
	\leq
	\tfrac{\overline{h}_n^3}{24}\bigl\|\eF_t(u)^{(3)}\bigr\|
	\fall t\in\I',\,u\in U_t. 
	\end{align*}
\end{exam}
\begin{exam}[cubic splines]
	The cubic splines $\beta_j^3:\R\to[0,1]$, $-3\leq j\leq n$, establish an ansatz space $X_n\subset C^2[a,b]$ of dimension $d_n=n+3$. Supplementing the interpolation conditions $(\pi_nu)(x_j)=u(x_j)$, $0\leq j\leq n$, by Hermite boundary conditions
	\begin{align*}
	(\pi_nu)''(a)&=u''(a),&
	(\pi_nu)''(b)&=u''(b)
	\end{align*}
	leads to a projection $\pi_n:C[a,b]\to X_n$ with $\norm{\pi_n}\leq 1+\tfrac{3}{2}\tfrac{\overline{h}_n}{\underline{h}_n}$ (cf.\ \cite[Thm.~3.1]{neuman:79}). Whence, $p:=1+\tfrac{3}{2}\sup_{n\in\N}\tfrac{\overline{h}_n}{\underline{h}_n}<\infty$ implies stability. Furthermore, the interpolation error becomes
	\begin{align*}
		\norm{u-\pi_nu}
		\leq
		\tfrac{5\overline{h}_n^4}{384}\bigl\|u^{(4)}\bigr\|
		\fall n\in\N,\,u\in C^4[a,b]
	\end{align*}
	(see \cite[Thm.~4]{hall:meyer:76}). In case \eqref{ide0} is $C^4[a,b]$-smoothing one obtains 
		\begin{align*}
		e_t^n(u)
		\leq
		\tfrac{5\overline{h}_n^4}{384}\bigl\|\eF_t(u)^{(4)}\bigr\|
		\fall t\in\I',\,u\in U_t
	\end{align*}
	as local discretization error. 
\end{exam}

In conclusion, collocation methods \eqref{iden}, $n\in\N$, based on splines having the order $l\in\set{1,2,3}$ are convergent, provided
\begin{itemize}
	\item $\lim_{n\to\infty}\overline{h}_n=0$ and

	\item the derivatives $\eF_t(\cdot)^{(l+1)}:U_t\to C[a,b]^d$ are bounded functions, that is, they map bounded subsets of $U_t$ to bounded sets uniformly in $t\in\I'$. 
\end{itemize}
\subsection{Pullback attractors}
Let $\I$ be unbounded below. An immediate consequence of \pref{propabc} is
\begin{cor}\label{corpb}
	The pullback absorbing set $\cA\subseteq\cU$ from \pref{propabc}(a) is positively invariant for every \eqref{iden}, $n\in\N_0$. 
\end{cor}
\begin{proof}
	Let $n\in\N_0$, $\tau\in\I'$ and $u\in u\cap\bar B_{\rho R_\tau}(0)$. Thanks to $(H_2)$ and $(H_4)$ it remains to establish the inclusion $\eF_\tau^n(u)\subseteq\bar B_{\rho R_{\tau+1}}(0)$. Thereto, we obtain as in the proof of \pref{propabc} that
	$$
		\norm{\eF_\tau^n(u)}
		\leq
		pb_\tau+pa_\tau\norm{u}
		\leq
		pb_\tau+\rho p a_\tau R_\tau.
	$$
	In combination with
	\begin{align*}
		R_{\tau+1}
		&\stackrel{\eqref{propabc1}}{=}
		p\sum_{l_1=-\infty}^{\tau}b_{l_1}\prod_{l_2=l_1+1}^{\tau}pa_{l_2}
		=
		pa_\tau p\sum_{l_1=-\infty}^{\tau-1}b_{l_1}\prod_{l_2=l_1+1}^{\tau-1}pa_{l_2}+pb_\tau\prod_{l_2=\tau+1}^\tau pa_{l_2}\\
		&\stackrel{\eqref{propabc1}}{=}
		pa_\tau R_\tau+pb_\tau
	\end{align*}
	this implies $\norm{\eF_\tau^n(u)}\leq\rho R_{\tau+1}+p(1-\rho)b_\tau\leq\rho R_{\tau+1}$. This is the assertion. 
\end{proof}
\begin{thm}[Pullback attractors for \eqref{iden}]\label{thmpb}
	Let $(H_1$--$H_4)$ and the limit relations \eqref{propabc1} hold for all $\tau\in\I$. If $\sup_{\tau\in\I}R_\tau<\infty$, then every \eqref{iden}, $n\in\N_0$, has a pullback attractor $\cA_n^\ast\subseteq\I\tm C_d$. If $\eF_t:U_t\to U_{t+1}$ maps bounded subsets of $U_t$ into bounded sets uniformly in $t\in\I'$ and moreover 
	\begin{itemize}
		\item[(i)] $(H_5)$ is satisfied with $\ell(r):=\sup_{t\in\I'}\ell_t(r)<\infty$ for all $r>0$, 

		\item[(ii)] the collocation discretizations \eqref{iden}, $n\in\N$, are convergent 
	\end{itemize}
	hold, then $\lim_{n\to\infty}\max_{t\in I}\dist{\cAn^\ast(t),\cA_0^\ast(t)}=0$ on each bounded $I\subseteq\I$.
\end{thm}
\begin{proof}
	Let $\rho>1$. We apply \tref{thmpbatt} and \ref{thmuscpbatt} to the difference equations \eqref{iden}, $n\in\N_0$, in the space $X=C_d$ equipped with the metric $d(u,v):=\norm{u-v}$. 

	By \pref{propabc}(a) the nonautonomous set $\cA:=\set{(\tau,u)\in\cU:\,\norm{u}\leq\rho R_\tau}$ is pullback absorbing for each \eqref{iden}, $n\in\N_0$. Furthermore, since the right-hand sides $\eF_t:U_t\to\Cd$ of \eqref{ide0} are completely continuous, also their discretizations $\eF_t^n=\Pi_n\eF_t$ share this property for all $t\in\I'$ and $n\in\N$. Therefore, \eqref{iden} is pullback asymptotically compact (see \cite[p.~13, Cor.~1.2.22]{poetzsche:10}) and \tref{thmpbatt} guarantees that every IDE \eqref{iden}, $n\in\N_0$, has a pullback attractor $\cA_n^\ast$ with the claimed properties.

	We verify the assumptions of \tref{thmuscpbatt}: 
	
	\emph{ad (i)}: 	By the Arzel{\`a}-Ascoli theorem \cite[p.~44, Thm.~3.3]{hirsch:lacombe:1999} we have to show that the union $\bigcup_{n\in\N_0}\cA_n^\ast(t)$ is bounded and equicontinuous for all $t\in\I$. Since boundedness is a consequence of the inclusion $\bigcup_{n\in\N_0}\cA_n^\ast(t)\subseteq\cA(t)$, it remains to establish the equicontinuity. 
	Let $n\in\N_0$. Since pullback attractors $\cA_n^\ast$ are invariant, one has
	$$
		\cA_n^\ast(t+1)=\eF_t^n(\cA_n^\ast(t))\subseteq\eF_t^n(\cA(t))\fall t\in\I'
	$$
	and consequently, $\bigcup_{n\in\N_0}\cA_n^\ast(t+1)\subseteq\bigcup_{n\in\N_0}\eF_t^n(\cA(t))$ for all $t\in\I'$. Because the subsets of equicontinuous sets are equicontinuous again, it suffices to establish equicontinuity of the union $\bigcup_{n\in\N_0}\eF_t^n(\cA(t))$ for any $t\in\I'$. 
	
	%
	
	Thereto, let $\eps>0$. Because \eqref{iden}, $n\in\N$, is assumed to be convergent, there exists an $N\in\N$ such that
	$$
		e^n_t(u)=\|\eF^n_t(u)-\eF^0_t(u)\|<\tfrac{\eps}{4} \fall u\in \cA(t),\,n\geq N.
	$$
	On the other hand, the image $\eF^0_t(\cA(t))$ due to \cite[p.~43, Prop.~3.1]{hirsch:lacombe:1999} is even uniformly equicontinuous, i.e., there exists a $\delta>0$ such that
	$$
		|x-y|<\delta
		\quad\Rightarrow\quad
		|\eF^0_t(u)(x)-\eF^0_t(u)(y)|<\tfrac{\eps}{4}\fall x,y\in\Omega.
	$$
	Combining both inequalities above, in case $|x-y|<\delta$ it results from the triangle inequality that
	\begin{align*}
		|\eF^n_t(u)(x)-\eF^n_t(u)(y)|
		\leq &\,
		|\eF^n_t(u)(x)-\eF^0_t(u)(x)| + |\eF^0_t(u)(x)-\eF^0_t(u)(y)|\\
		\leq &\,
		\tfrac{\eps}{2} + \tfrac{\eps}{4}<\eps
		\fall u\in\cA(t),\, n\geq N.
	\end{align*}
	This yields that $\eF_t^0(\cA(t))\cup\bigcup_{n\geq N}\eF_t^n(\cA(t))$ is equicontinuous and additionally including the finitely many relatively compact sets $\eF_t^n(\cA(t))$, $1\leq n<N$, preserves equicontinuity. As a result, every $\bigcup_{n\in\N_0}\cA_n^\ast(t)$, $t\in\I$, is relatively compact. 
	
	\emph{ad (ii)}: Due to $\bigcup_{n\in\N_0}\bigcup_{s\leq t}\cA_n^\ast(s)\subseteq\cA(t)$ for all $t\in\I$ we obtain the claim. 

	\emph{ad (iii)}: In order to apply \pref{properr} we note our assumption $(H_5)$ and \lref{lemlip} imply that for each $r>0$ there exists a $\ell(r)\geq 0$ such that
	$$
		\norm{\eF_t^n(u)-\eF_t^n(v)}
		\leq
		p\ell(r)\norm{u-v}\fall n\in\N,\,t\in\I'
	$$
	and $u,v\in U_t\cap\bar B_r(0)$. Hence, the collocation discretizations \eqref{iden} fulfill \eqref{properror1}. Now let $\eps>0$, $\tau\in\I'$, $S\in\N$ and $K\subseteq\Cd$ be compact. Given $u\in U_{\tau-s}\cap K$ choose a nonautonomous set $\cB\subseteq\cU$ such that
	$$
		\vphi^0(\tau;\tau-s,u)\in\cB(\tau)\fall 0<s\leq S.
	$$
	Because, by assumption, $\eF_t:U_t\to U_{t+1}$ maps bounded subsets of $U_t$ into bounded sets uniformly in $t\in\I'$, the set $\cB$ can be chosen to be bounded. Hence, there exists a $r>0$ such that $\cB(t)\subseteq\bar B_r(0)$ for all $t\in\I$. Now choose $\rho>r$ and due to \pref{properr} there exists an $N_0\in\N$ such that 
	\begin{align*}
		\norm{\vphi^n(\tau;\tau-s,u)-\vphi^0(\tau;\tau-s,u)}
		\stackrel{\eqref{properr3}}{\leq}&
		\Gamma\left(\tfrac{1}{n}\right)\sum_{l_1=\tau-s}^{\tau-1}C(\cB(l_1))\prod_{l_2=l_1+1}^{\tau-1}p\ell(\rho)\\
		\leq&
		\Gamma\left(\tfrac{1}{n}\right)C(\bar B_r(0))\sum_{l_1=\tau-s}^{\tau-1}\prod_{l_2=l_1+1}^{\tau-1}p\ell(\rho) 
	\end{align*}
	for all $0<s\leq S$ and $n>N_0$. It is easy to see that the right-hand side in the above estimate does not depend on $\tau\in\I'$. Since our convergence assumption (ii) ensures $\lim_{n\to\infty}\Gamma(\tfrac{1}{n})=0$, there exists an $n_0\geq N_0$ such that $n>n_0$ implies
	$$
		\norm{\vphi^n(\tau;\tau-s,u)-\vphi^0(\tau;\tau-s,u)}<\eps
		\fall u\in U_{\tau-s}\cap K,\,0<s\leq S.
	$$
	This finally verifies \tref{thmuscpbatt}(iii). 
\end{proof}
\subsection{Forward dynamics}
Let all nonempty, closed $Z:\equiv Z_t$ be time-invariant yielding the constant domains $U:\equiv U_t$ on a discrete interval $\I$ unbounded above.

The forward dynamics counterpart to \cref{corpb} has a slightly different structure: 
\begin{cor}\label{corfw}
	The forward absorbing set $\cA\subseteq\cU$ from \pref{propabc}(b) has constant fibers, i.e.\ there exists a $R>0$ with $R=R_\tau$ for all $\tau\in\I$. Moreover, in case
	\begin{equation}
		pb_\tau\leq\rho R(1-pa_\tau)\fall\tau\in\I
		\label{corfw1}
	\end{equation}
	the nonautonomous set $\cA$ is positively invariant for every \eqref{iden}, $n\in\N_0$. 
\end{cor}
\begin{proof}
	Above all, for $\tau,t\in\I$ we define $r_\tau(t):=p\sum_{l_1=\tau}^{t-1}b_{l_1}\prod_{l_2=l_1+1}^{t-1}pa_{l_2}$ and passing to the limit $t\to\infty$ in
	$$
		r_\tau(t)
		=
		p\sum_{l_1=\tau+1}^{t-1}b_{l_1}\prod_{l_2=l_1+1}^{t-1}pa_{l_2}
		+
		pb_\tau\prod_{l_2=\tau+1}^{t-1}pa_{l_2}
		=
		r_{\tau+1}(t)
		+
		pb_\tau\prod_{l_2=\tau+1}^{t-1}pa_{l_2}
	$$
	implies that $R_\tau=R_{\tau+1}$ due to \eqref{propabc2}. With $R:=R_\tau$ and $u\in u\cap\bar B_{\rho R}(0)$ we obtain as in the proof of \cref{corpb} that $\norm{\eF_\tau^n(u)}\leq pb_\tau+\rho p a_\tau R$ for all $n\in\N_0$ and therefore \eqref{corfw1} yields the assertion. 
\end{proof}

\begin{thm}[Forward limit sets for \eqref{iden}]\label{thmconv}
	Suppose $(H_1$--$H_4)$, the limit relations \eqref{propabc2} and \eqref{corfw1} hold for all $\tau\in\I$ yielding the nonautonomous set $\cA\subseteq\cU$ from \pref{propabc}(b). If each \eqref{iden} is strongly $\cA$-asymp\-to\-ti\-cally compact with a compact $K_n\subseteq U$, $n\in\N_0$, then its forward limit set $\omega_{\cA,n}^+\subseteq K_n$ is nonempty, compact and forward attracts $\cA$. If moreover $K:=\bigcup_{n\in\N_0}K_n$ is bounded, $\eF_t:U\to U$ maps bounded subsets of $U$ into bounded sets uniformly in $t\in\I$ and 
	\begin{itemize}
		\item[(i)] $\omega_{\cA,0}^+$ attracts $K$ uniformly in $\tau\in\I$, i.e.,
		$$
			\lim_{s\to\infty}\sup_{\tau\in\I}\dist{\vphi^0(\tau+s;\tau,K),\omega_{\cA,0}^+}=0, 
		$$

		\item[(ii)] for every sequence $\intoo{(s_k,\tau_k)}_{k\in\N}$ in $\N\tm\I$ with $\tau_k\to\infty$ as $k\to\infty$ one has
		$$
			\lim_{k\to\infty}\dist{\vphi^n(\tau_k+s_k;\tau_k,K_n),K_n}=0\fall n\in\N_0,
		$$
		
		\item[(iii)] $(H_5)$ is satisfied with $\ell(r):=\sup_{t\in\I}\ell_t(r)<\infty$ for all $r>0$,
		
		\item[(iv)] the collocation discretizations \eqref{iden}, $n\in\N$, are convergent 
	\end{itemize}
	hold, then $\lim_{n\to\infty}\dist{\omega_{\cA,n}^+,\omega_{\cA,0}^+}=0$. 
\end{thm}
\begin{proof}
	We aim to apply \tref{thmusfw} to \eqref{iden}, $n\in\N_0$. Above all, it results from the above \cref{corfw} that $\cA$ is forward absorbing and positively invariant for each \eqref{iden}. Then \tref{thmusfw} yields the existence of forward limit sets $\omega_{\cA,n}^+$ with the claimed properties for all $n\in\N_0$. 

	It remains to show that the assumptions (iii) and (iv) of \tref{thmusfw} hold. 
	
	\emph{ad (iii)}: Choose $r>0$ so large that $\cA(\tau)\cup K\subseteq\bar B_r(0)$ holds for all $\tau\in\I$, which is possible due to the boundedness of $\cA$ and $K$. By assumption, it results inductively from \lref{lemlip} that
	$$
		\norm{\vphi^n(\tau+s;\tau,u)-\vphi^n(\tau+s;\tau,v)}
		\leq
		(p\ell(r))^s\norm{u-v}\fall n\in\N_0,\,s\in\N
	$$
	and $\tau\in\I$, $u,v\in\cA(\tau)\cup K_n$. Let $\eps>0$ and $S\in\N$ be arbitrarily given. If we choose $\delta:=\min_{0<s\leq S}\tfrac{\eps}{\ell(r)^s}>0$, then $\norm{u-v}<\delta$ readily leads to the estimate $\norm{\vphi^n(\tau+s;\tau,u)-\vphi^n(\tau+s;\tau,v)}<\eps$ for $0<s\leq S$, which implies (iii). 

	\emph{ad (iv)}: Let $\eps>0$, $S\in\N$ and $\tau\in\I$. As in the proof of \tref{thmpb} one shows that the Lipschitz estimate \eqref{properr1} holds in our present setting with the (local) Lipschitz constant $p\ell(r)$. For $u\in\cA(\tau)\cup K$ choose a bounded nonautonomous set $\cB\subseteq\cU$ so large that $\vphi^0(t;\tau,u)\in\cB(t)$ for all $\tau\leq t\leq\tau+S$; since $\eF_t$ is assumed to be bounded uniformly in $t\in\I$ the nonautonomous set $\cB$ can be chosen to be independent of $\tau\in\I$. Furthermore, choose $r>0$ so large that $\cB(t)\subseteq\bar B_r(0)$ for all $\tau\leq t\leq\tau+S$. For $\rho>r$, then \pref{properr} implies that there is an $N_0\in\N$ with
	\begin{align*}
		\norm{\vphi^n(\tau+s;\tau,u)-\vphi^0(\tau+s;\tau,u)}
		\stackrel{\eqref{properr3}}{\leq}&
		\Gamma\left(\tfrac{1}{n}\right)\sum_{l_1=\tau}^{\tau+s-1}C(\cB(l_1))\prod_{l_2=l_1+1}^{\tau+s-1}\ell(\rho)\\
		\leq&
		\Gamma\left(\tfrac{1}{n}\right)C(\bar B_r(0))
		\sum_{l_1=\tau}^{\tau+s-1}\prod_{l_2=l_1+1}^{\tau+s-1}\ell(\rho)
	\end{align*}
	for all $0<s\leq S$ and $n>N_0$. Note that the right-hand side of this inequality does not depend on $\tau\in\I$. Consequently, our convergence condition $\lim_{n\to\infty}\Gamma(\tfrac{1}{n})=0$ implies that there exists an $n_0\geq N_0$ such that $n>n_0$ implies
	$$
		\norm{\vphi^n(\tau+s;\tau,u)-\vphi^0(\tau+s;\tau,u)}
		<
		\eps\fall u\in\cA(\tau)\cap K,\,0<s\leq S
	$$
	and therefore \tref{thmusfw}(iii) is verified. 
\end{proof}

\begin{thm}[Persistence of forward attractors for \eqref{iden}]
	Let $\I=\Z$. Suppose $(H_1$--$H_4)$, the limit relations \eqref{propabc2} and \eqref{corfw1} hold for all $\tau\in\Z$ yielding the nonautonomous set $\cA\subseteq\cU$ from \pref{propabc}(b). If each \eqref{iden} is strongly $\cA$-asymp\-to\-ti\-cally compact with a compact $K_n\subseteq U$, $n\in\N_0$, and 
	\begin{itemize}
		\item[(i)] for every sequence $\intoo{(s_k,\tau_k)}_{k\in\N}$ in $\N\tm\Z$ with $\tau_k\to\infty$ as $k\to\infty$ one has
		$$
			\lim_{k\to\infty}\dist{\vphi^n(\tau_k+s_k;\tau_k,K_n),K_n}=0\fall n\in\N_0,
		$$
		
		\item[(ii)] $(H_5)$ is satisfied with $\ell(r):=\sup_{t\in\I}\ell_t(r)<\infty$ for all $r>0$,
		
		\item[(iii)] the forward limit sets $\omega_{\cA,n}^+$ of \eqref{iden} satisfy
		$$
			\omega_{\cA,n}^+=\limsup_{t\to\infty}\cA_n^\ast(t)\fall n\in\N_0
		$$
		with $\cA_n^\star(t):=\bigcap_{0\leq s}\vphi^n(t;t-s,\cA(t-s))$ for all $t\in\Z$
	\end{itemize}
	hold, then every \eqref{iden}, $n\in\N_0$, has $\cA_n^\star\subseteq\cA$ as a forward attractor. 
\end{thm}
\begin{proof}
	Let $n\in\N_0$. We verify that \eqref{iden} satisfies the assumptions of \tref{thmfwatt}. First of all, \cref{corfw} ensures that the closed set $\cA$ is forward absorbing and positively invariant w.r.t.\ \eqref{iden}. Because each $\eF_t:U\to\Cd$ is completely continuous, also the discretizations $\eF_t^n=\Pi_n\eF_t$ have this property for all $t\in\Z$. Therefore, \eqref{iden} is pullback asymptotically compact (see \cite[p.~13, Cor.~1.2.22]{poetzsche:10}). Our assumption (i) yields that \eqref{iden} fulfills the assumption in \tref{thmfwatt}(i). As in the above proof of \tref{thmconv} one shows that also \tref{thmfwatt}(ii) holds for \eqref{iden}. Finally, (iii) ensures that finally the assumption in \tref{thmfwatt}(iii) is satisfied for \eqref{iden}. 
\end{proof}
\section{Illustrations}
\label{sec4}
Let $\I$ be a discrete interval and $\Omega\subset\R^\kappa$ be a compact habitat. We are discussing IDEs defined on the (closed) cone
$$
	U=C_+(\Omega):=\set{u\in C(\Omega):\,0\leq u(x)\text{ for all }x\in\Omega}
$$
of nonnegative continuous functions. 
\subsection{Pullback attractors for generalized Beverton-Holt equations}
Consider the scalar nonautonomous IDE
\begin{align} \label{BH}
u_{t+1}&=\eF_t(u_t),&
\eF_t(u)&:=\int_\Omega k_t(\cdot,y)\frac{\gamma_t(y)u(y)}{1+u(y)^\alpha}\,{\mathrm d} y
\end{align}
depending on a parameter $\alpha>0$ with continuous kernels $k_t:\Omega^2\to\R_+$ and integrable growth rates $\gamma_t:\Omega\to\R_+$, $t\in\I'$.
\begin{SCfigure}[2]
	\includegraphics[scale=0.5]{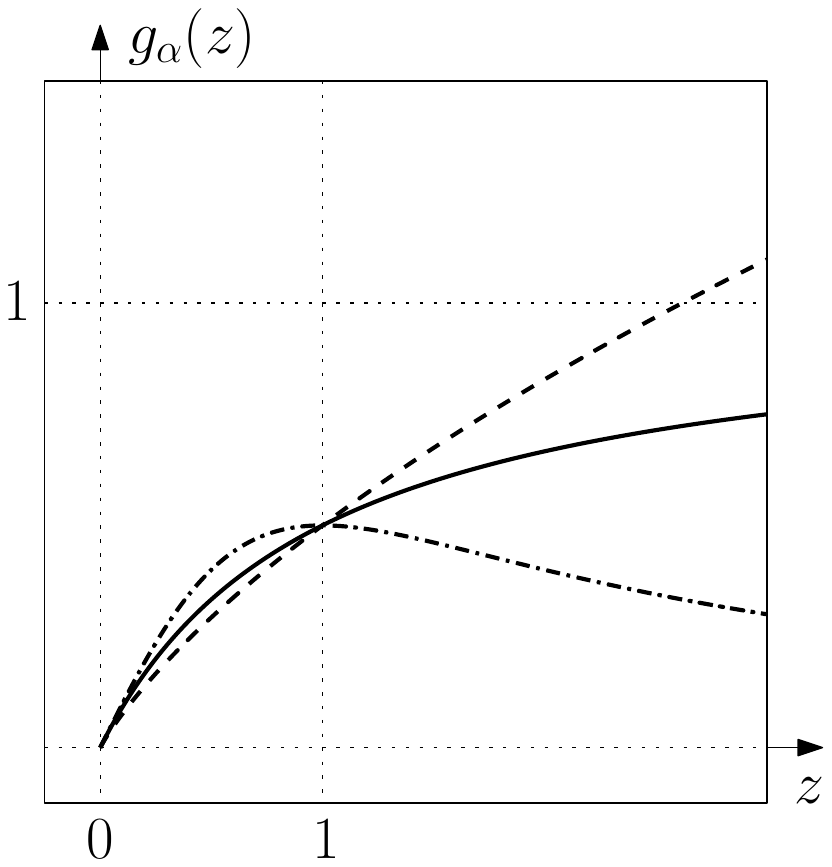}
	\caption{Shape of generalized Bev\-er\-ton-Holt functions $g_\alpha(z):=\tfrac{z}{1+z^\alpha}$ for $\alpha=\tfrac{1}{2}$ (dashed), $\alpha=1$ (solid) and $\alpha=2$ (dash dotted)}
	\label{figbh}
\end{SCfigure}

The nonlinearity in this Hammerstein IDE is given by the family of generalized Beverton-Holt functions $g_\alpha:\R_+\to\R_+$, $g_\alpha(z):=\tfrac{z}{1+z^\alpha}$, whose behavior depends on $\alpha>0$ (see \fref{figbh}). All these growth functions have in common to be globally Lipschitz with constant $1$. Consequently, in order to fulfill the assumptions of \lref{lemlip} we choose $\lambda_t(x,y):=k_t(x,y)\gamma_t(y)$ guaranteeing even a global Lipschitz estimate in \eqref{lemlip2} with
$$
	\ell_t:=\sup_{x\in\Omega}\int_\Omega k_t(x,y)\gamma_t(y)\,{\mathrm d} y
	\fall t\in\I'. 
$$
In order to establish that the generalized Beverton-Holt equation \eqref{BH} is dissipative using \pref{propabc}, we suppose that $(\ell_t)_{t\in\I'}$ is bounded and proceed as follows: 
\begin{itemize}
	\item For $\alpha\in(0,1)$ the function $g_\alpha$ is unbounded and its tangent in a point $(\zeta,g_\alpha(\zeta))$ for $\zeta>0$ reads as
$$
	z
	\mapsto
	\tfrac{1+(1-\alpha)\zeta^\alpha}{(1+\zeta^\alpha)^2}(z-\zeta)+\tfrac{\zeta}{1+\zeta^\alpha}
	= 
	\tfrac{\alpha\zeta^{1+\alpha}}{(1+\zeta^\alpha)^2}
	+
	\tfrac{1+(1-\alpha)\zeta^\alpha}{(1+\zeta^\alpha)^2}z.
$$
Concerning the estimate \eqref{assumpK0} this allows us to choose
\begin{align*}
	\alpha_t(x,y)&:=\tfrac{1+(1-\alpha)\zeta^\alpha}{(1+\zeta^\alpha)^2}k_t(x,y)\gamma_t(y),&
	\beta_t(x,y)&:=\tfrac{\alpha\zeta^{1+\alpha}}{(1+\zeta^\alpha)^2}k_t(x,y)\gamma_t(y)
\end{align*}
	and consequently
	\begin{align*}
		a_t&=\tfrac{1+(1-\alpha)\zeta^\alpha}{(1+\zeta^\alpha)^2}\ell_t,&
		b_t&=\tfrac{\alpha\zeta^{1+\alpha}}{(1+\zeta^\alpha)^2}\ell_t.
	\end{align*}
	If $\zeta>0$ is chosen so large that 
	$$
		\tfrac{1+(1-\alpha)\zeta^\alpha}{(1+\zeta^\alpha)^2}\limsup_{t\to-\infty}\ell_t
		<
		1\quad\text{on $\I$ unbounded below}, 
	$$
	then \eqref{BH} is pullback dissipative, while 
	$$
		\tfrac{1+(1-\alpha)\zeta^\alpha}{(1+\zeta^\alpha)^2}\limsup_{t\to\infty}\ell_t
		<
		1\quad\text{on $\I$ unbounded above}
	$$
	yields forward dissipativity of the Beverton-Holt equation \eqref{BH}. 
	
	\item For $\alpha\geq 1$, the function $g_\alpha$ is globally bounded with
$$
	\sup_{z\geq 0}g_\alpha(z)
	=
	\tfrac{1}{\alpha}(\alpha-1)^{1-\tfrac{1}{\alpha}}
$$
and therefore in \eqref{assumpK0} we can choose
\begin{align*}
	\alpha_t(x,y)&:\equiv 0,&
	\beta_t(x,y)&:=\tfrac{1}{\alpha}(\alpha-1)^{1-\tfrac{1}{\alpha}}k_t(x,y)\gamma_t(y),
\end{align*}
	which results in
	\begin{align*}
		a_t&\equiv 0,&
		b_t&=\tfrac{1}{\alpha}(\alpha-1)^{1-\tfrac{1}{\alpha}}\ell_t.
	\end{align*}
	Consequently, the IDE \eqref{BH} is both forward and pullback dissipative. 
\end{itemize}
\begin{figure}
	\includegraphics[trim=0.65cm 0.65cm 0.65cm 0.65cm,clip,width=0.325\linewidth]{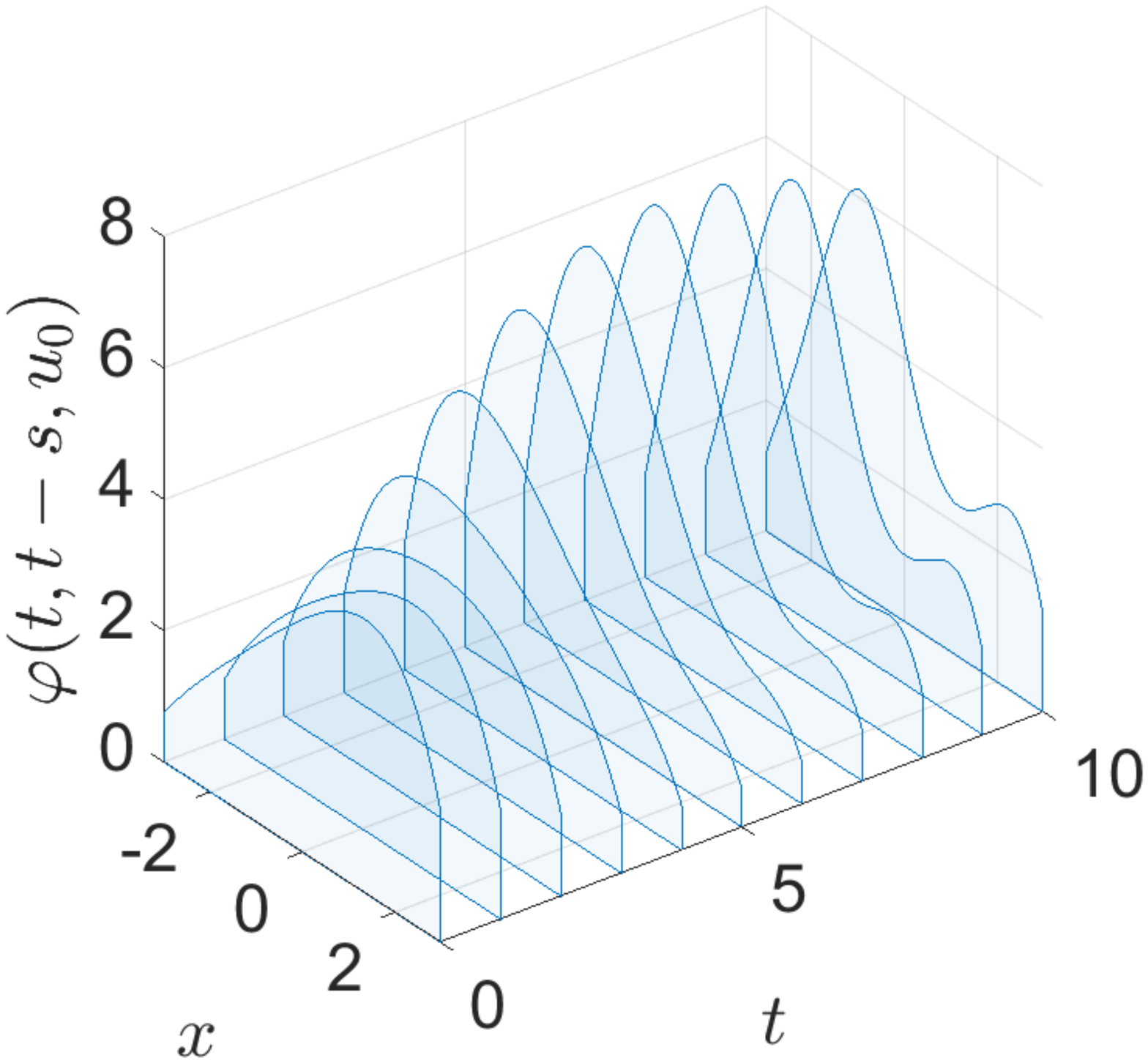}
	\includegraphics[trim=0.65cm 0.65cm 0.65cm 0.65cm,clip,width=0.325\linewidth]{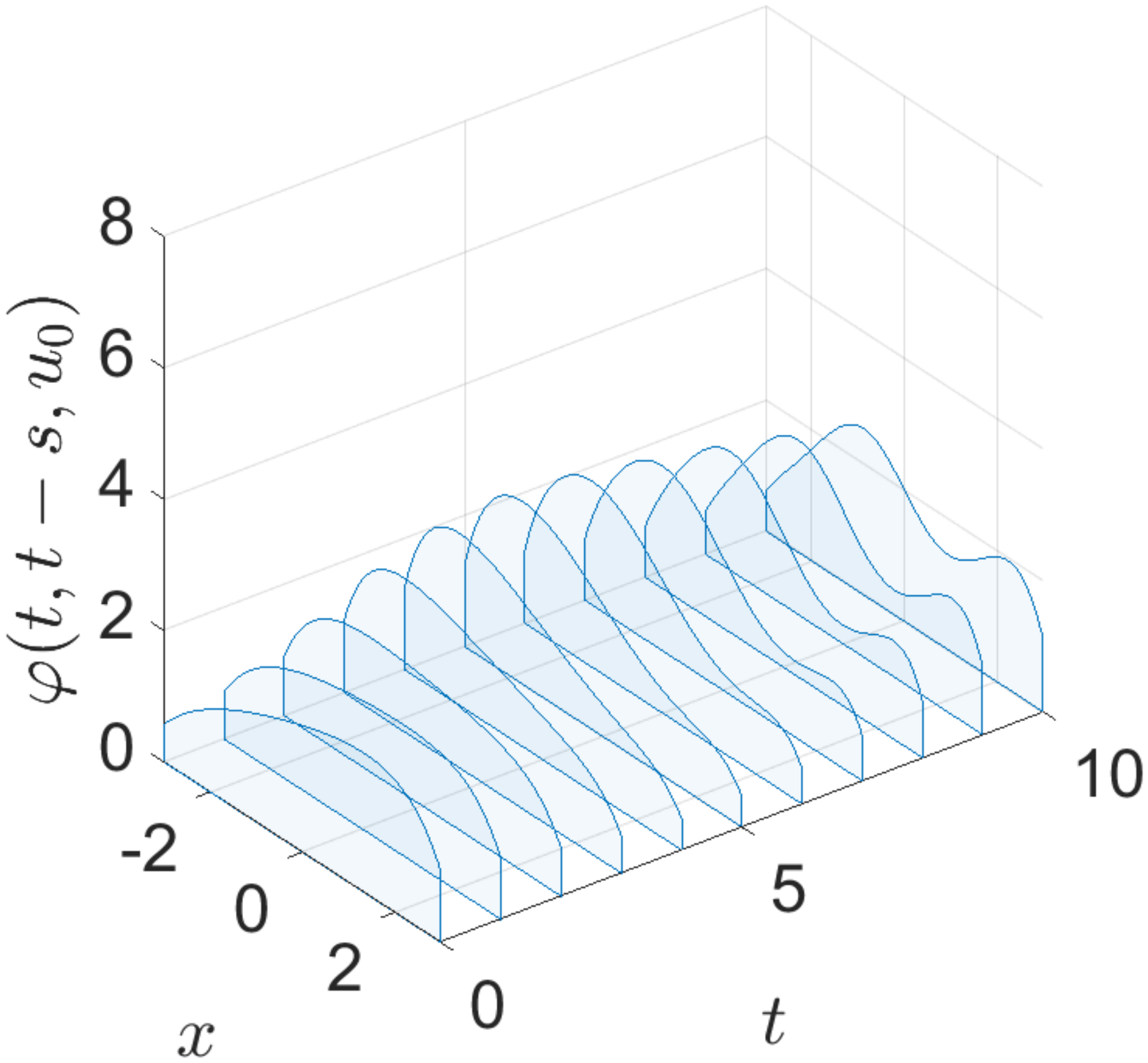}
	\vspace{-20pt}
	\caption{Sequences of sets containing functions in the fibers $\cA_n^\ast(t)$, $0\leq t\leq 10$, for $\alpha\in\set{\tfrac{1}{2},1}$, respectively from left to right.}
	\label{figpb}
\end{figure}
Now let $\I$ be unbounded below. In any case, we conclude from \tref{thmpbatt} that the nonautonomous Beverton-Holt equation \eqref{BH} has a pullback attractor $\cA_0^\ast$. In particular, for parameters $\alpha\in(0,1]$ the growth function $g_\alpha$ is strictly increasing and therefore the right-hand sides of \eqref{BH} are order-preserving. This means that for all $u,v\in U$ the implication
$$
	u(x)\leq v(x)
	\quad\Rightarrow\quad
	\eF_t^n(u)(x)\leq\eF_t(v)(x)\fall t\in\I',\,x\in\Omega
$$
holds. Given a sequence $(\xi_t)_{t\in\I}$ in $C_+(\Omega)$ we introduce the order intervals
$$
	[0,\xi]
	:=
	\set{(t,v)\in\I\tm C(\Omega):\,0\leq v(x)\leq\xi_t(x)\text{ for all }x\in\Omega}
	\subseteq
	\I\tm C_+(\Omega)
$$
and obtain from \cite[Prop.~8]{poetzsche:15} that the pullback attractor $\cA_0^\ast$ of \eqref{BH} satisfies the inclusion
$
	\cA_0^\ast
	\subseteq
	[0,\xi^0]
$
with a bounded entire solution $(\xi_t^0)_{t\in\I}$ to \eqref{BH}. This entire solution is contained in $\cA_0^\ast$ due to the characterization \eqref{besol}. In addition, $\xi^0$ is the pullback limit for all sequences starting from 'above' $\xi^0$ (cf.~\cite[Prop.~8]{poetzsche:15}). 

Based on \tref{thmpb} also stable collocation discretizations \eqref{iden}, $n\in\N$, of the generalized Beverton-Holt equation \eqref{BH} possess pullback attractors $\cA_n^\ast$. In order to illustrate convergence of the pullback attractors $\cA^\ast_n$ to $\cA_0^\ast$ as $n\to\infty$, we restrict to habitats $\Omega=[a,b]$ and piecewise linear splines from \eref{exp1} having the nodes $x_j:=a+j\tfrac{b-a}{n}$, $0\leq j\leq n$, for all $n\in\N$. This results in the spatial discretizations
\begin{align}
	u_{t+1}&=\eF_t^n(u_t),&
	\eF_t^n(u)&=
	\sum_{j=0}^n\int_\Omega k_t(x_j,y)\frac{\gamma_t(y)u(y)}{1+u(y)^\alpha}\,{\mathrm d} y\beta_{j-1}^1.
	\label{BHn}
\end{align}
In particular, $p=1$ and $\overline{h}_n=\tfrac{b-a}{n}$ for all $n\in\N$ and for appropriate kernels $k_t$ one obtains the convergence rate $2$ from \eqref{b1est}. 

Spatial discretization using piecewise linear functions guarantees that also the right-hand sides $\eF_t^n:U\to X_n$, $t\in\I'$, are order preserving for $\alpha\in(0,1]$. Thus, \cite[Prop.~8]{poetzsche:15} applies to \eqref{BHn} as well, and ensures that its pullback attractors fulfill 
$
	\cA^\ast_n
	\subseteq
	[0,\xi^n]
$
with bounded entire solutions $(\xi_t^n)_{t\in\I}$ to \eqref{BHn} in $\cA_n^\ast$. 
\begin{table}
\begin{center}
\begin{tabular}{ r || c | c }
	$n$ & $\alpha=\tfrac{1}{2}$ & $\alpha=1$ \\
	\hline \rule{0pt}{15pt}
	16 & 2.112614126300029 & 2.100856100109834 \\
	32 & 2.055209004601208 & 2.051123510423984 \\
	64 & 2.026777868073563 & 2.025681916479720 \\
	128 & 2.013096435137189 & 2.012865860858589 \\
	256 & 2.006536458546063 & 2.006433295172438 \\
	512 & 2.003256451919377 & 2.003220576642864 \\
	1024 & 2.001624177537549 & 2.001610772707442
\end{tabular}
\caption{Values of the approximate convergence rates $c(n)$ for $n\in\set{2^4,\ldots,2^{10}}$ nodes and parameters $\alpha\in\set{\tfrac{1}{2},1,2}$.}
\label{tabpb1}
\end{center}
\end{table}

On this basis, in order to quantify convergence rates of the Hausdorff distances $\dist{\cA_n^\ast(t),\cA_0^\ast(t)}$ between the fibers of the pullback attractors to \eqref{BHn} and $\cA_0^\ast$ guaranteed by \tref{thmpb}, we make use the entire solutions $\xi^n$ and approximate the rates of their convergence to $\xi^0$ as $n\to\infty$. For this purpose, choose the habitat $\Omega=[-3,3]$, the Laplace kernel $k_t(x,y):=\tfrac{\delta_t}{2}e^{-\delta_t|x-y|}$ yielding $C^2[-3,3]$-smoothness in the right-hand side of \eqref{BH} for the almost periodic dispersal rates $\delta_t:=2+\sin(\tfrac{t}{3})$, the growth rates $\gamma_t(x):=3-\sin(\tfrac{tx}{5})$ and $n=4096$. In order to evaluate the remaining integrals in \eqref{BHn} we apply the trapezoidal rule and the functions $\xi_t^n$ are approximated as pullback limits $\vphi^n(t;t-s,\xi)$, $t\in\I$, with $s:=15$ and upper solutions $\xi$. Then the sequence $c(n):=\log_2\tfrac{\|\xi^n-\xi^{2n}\|}{\|\xi^{2n}-\xi^{4n}\|}$ approximates the desired convergence rates for $\xi^n$ to $\xi^0$ given large values of $n$. This results in the values listed in Tab.~\ref{tabpb1} which where illustrated in \fref{convratepb}, respectively. Clearly, the rate $2$ (cf.~\eref{exp1}) is preserved. 
\begin{figure}
\vspace{-80pt}
\includegraphics[trim=0.5cm 0.5cm 0.5cm 0.5cm,clip,width=0.65\linewidth]{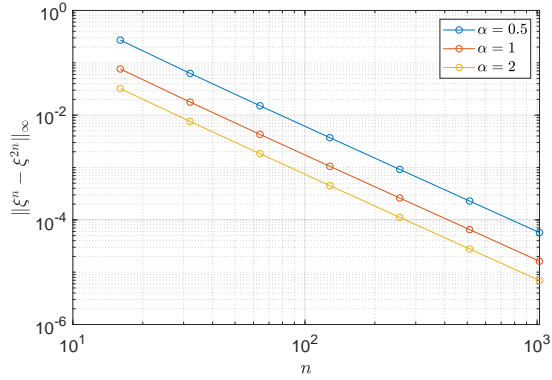}
\vspace{-90pt}
\caption{Development of the averaged error $\|\xi^n-\xi^{2n}\|$ for $n\in\set{2^4,\ldots,2^{10}}$ nodes and parameters $\alpha\in\set{\tfrac{1}{2},1,2}$.}
\label{convratepb}
\end{figure}
\subsection{Forward limit sets for asymptotically autonomous Ricker equations}
Consider the scalar nonautonomous IDE
\begin{align}
	u_{t+1}&=\eF_t(u_t),&
	\eF_t(u)&:=\gamma_t\int_\Omega k(\cdot,y)u_t(y)e^{-u_t(y)}\,{\mathrm d} y+b
	\label{exrick}
\end{align}
with a continuous kernel $k:\Omega\tm\Omega\to\R_+$, an inhomogeneity $b\in C_+(\Omega)$ and a bounded sequence of growth rates $(\gamma_t)_{t\in\I}$ in $\R_+$. For the Ricker growth function $g:\R_+\to\R_+$, $g(z):=ze^{-z}$ it is elementary to show
\begin{align*}
	\abs{g(z)}&\leq\tfrac{1}{e},&
	\abs{g(z)-g(\bar z)}&\leq\tfrac{1}{e^2}\abs{z-\bar z}\fall z,\bar z\in\R_+.
\end{align*}
Given this, on intervals $\I$ unbounded below it is not difficult to establish that \eqref{exrick} has a pullback attractor. However, we are interested in the forward dynamics of \eqref{exrick} under discretization. For this purpose, let us suppose that $\I$ is bounded below, but unbounded above. Moreover, it is crucial to assume that \eqref{exrick} is asymptotically autonomous in the sense that the coefficient sequence $(\gamma_t)_{t\in\I}$ converges exponentially. In order to be more precise, let us assume that there exist reals $\gamma>0$ and $\tilde{K}_0,\tilde{K}_1\geq 1$ such that the following holds with $k_0:=\sup_{x\in\Omega}\int_\Omega k(x,y)\,{\mathrm d} y$: 
\begin{itemize}
	\item $\tilde{K}_0:=\sup_{s\leq t}\prod_{l=s}^{t-1}\tfrac{\gamma_l}{\gamma}<\infty$ and $\gamma k_0<1$, 

	\item $\abs{\gamma_t-\gamma}\leq \tilde{K}_1(k_0\gamma)^t$, which implies $\lim_{t\to\infty}\gamma_t=\gamma$, 

	\item $k_0\sup_{t\in\I}\gamma_t<\tfrac{e^2}{1+e^2}\tfrac{1-k_0\gamma}{\tilde{K}_0}$.
\end{itemize}
As a result, it is shown in \cite[Exam.~5.6]{huynh:poetzsche:kloeden:20} that the autonomous limit equation
\begin{align}
	u_{t+1}&=\eF(u_t),&
	\eF(u)&:=\gamma\int_\Omega k(\cdot,y)u_t(y)e^{-u_t(y)}\,{\mathrm d} y+b
	\label{exricka}
\end{align}
has a singleton global attractor $A^\ast=\set{u^\ast}$ and for any bounded set $A\subseteq C_+(\Omega)$ absorbing for \eqref{exricka} the forward limit set of \eqref{exrick} arising from $\cA:=\I\tm A$ is just $\omega_{\cA,0}^+=\set{u^\ast}$. 

In addition, it can be shown using \cite[Thm.~4.14]{huynh:poetzsche:kloeden:20} that for every bounded subset $B\subseteq C_+(\Omega)$ one has the limit relation
$$
	\lim_{s\to\infty}\sup_{b\in B}\norm{\vphi^0(\tau+s;\tau,b)-\eF^s(b)}=0\quad\text{uniformly in }\tau\in\I;
$$
for this $\I$ needs to be bounded below. Consequently, passing to the least upper bound over $b\in B$ in the inequality
$$
	\inf_{a\in B}\norm{\vphi^0(\tau+s;\tau,b)-\eF^s(a)}
	\leq
	\norm{\vphi^0(\tau+s;\tau,b)-\eF^s(b)}\fall b\in B
$$
leads to the limit relation
$$
	0
	\leq
	\dist{\vphi^0(\tau+s;\tau,B),\eF^s(B)}
	\leq
	\sup_{b\in B}\norm{\vphi^0(\tau+s;\tau,b)-\eF^s(b)}
	\xrightarrow[s\to\infty]{} 0
$$
uniformly in the initial time $\tau\in\I$. Since the global attractor $A^\ast$ of \eqref{exricka} attracts bounded subsets $B\subseteq X$, one therefore arrives at
\begin{align}
	&
	\dist{\vphi^0(\tau+s;\tau,B),\omega_{\cA,0}^+}
	\notag\\
	\stackrel{\eqref{hs2}}{\leq}&
	\dist{\vphi^0(\tau+s;\tau,B),\eF^s(B)}+\dist{\eF^s(B),\omega_{\cA,0}^+}
	\notag\\
	=&
	\dist{\vphi^0(\tau+s;\tau,B),\eF^s(B)}+\dist{\eF^s(B),A^\ast}
	\xrightarrow[s\to\infty]{} 0
	\label{unilimest}
\end{align}
uniformly in $\tau\in\I$. 

Along with the asymptotically autonomous Ricker equation \eqref{exrick} we now turn to their convergent collocation discretizations \eqref{iden}, $n\in\N$, satisfying $(H_3$--$H_4)$. With the abbreviations
\begin{align*}
	R&:=\tfrac{pk_0}{e}\sup_{t\in\I}\gamma_t,&
	b_n&:=\pi_nb\in X_n
\end{align*}
one can show that the nonautonomous set
\begin{align*}
	\cA&:=\Bigl\{(t,x)\in\cU:\,u\in\overline{\bigcup_{n\in\N_0}\eF_{t^\ast}^n(U\cap B_R(b_n))}\Bigr\},&
	t^\ast&:=
	\begin{cases}
		t-1,&t>\min\I,\\
		t,&t=\min\I
	\end{cases}
\end{align*}
is bounded, compact, positively invariant and forward absorbing w.r.t.\ every \eqref{iden}, $n\in\N_0$, with absorption time $2$. In particular, each fiber $\cA(t)$, $t\in\I$, is compact, since every $\eF_{t^\ast}^n(U\cap B_R(b_n))$ is relatively compact and \eqref{iden}, $n\in\N$, is convergent. 

The above exponential convergence assumption for $(\gamma_t)_{t\in\I}$ implies that the union $\bigcup_{t\in\I}\eF_t^n(U\cap B_R(b_n))$ is relatively compact. Thanks to the Lipschitz estimate
$$
	\norm{\eF_t^n(u)-\eF_t^n(v)}
	\leq
	\tfrac{pk_0}{e^2}\sup_{t\in\I}\gamma_t\norm{u-v}
	\fall n\in\N_0,\,t\in\I,\,u,v\in U, 
$$
we can apply \cite[Thm.~5.5]{huynh:poetzsche:kloeden:20} in order to obtain that the forward limit sets $\omega_{\cA,n}^+$ of \eqref{iden}, $n\in\N$, are asymptotically negatively invariant (cf.\ the proof of \tref{thmusfw}). 

Because $K:=\bigcup_{n\in\N_0}K_n$ with $K_n:=\overline{\bigcup_{t\in\I}\eF_t^n(U\cap B_R(b_n))}$ is bounded, we obtain from the uniform limit relation \eqref{unilimest} established above that the assumption in \tref{thmconv}(i) is satisfied. Combining this with the asymptotic negative invariance of each $\omega_{\cA,n}^+$, $n\in\N_0$, we conclude as in \tref{thmconv} that the forward limit sets of the collocation discretizations \eqref{iden}, $n\in\N$, fulfill
$$
	\lim_{n\to\infty}\dist{\omega_{\cA,n}^+,\omega_{\cA,0}^+}
	=
	\lim_{n\to\infty}
	\sup_{a\in\omega_{\cA,n}^+}\norm{a-u^\ast}
	=
	0.
$$
\providecommand{\bysame}{\leavevmode\hbox to3em{\hrulefill}\thinspace}
\providecommand{\MR}{\relax\ifhmode\unskip\space\fi MR }
\providecommand{\MRhref}[2]{\href{http://www.ams.org/mathscinet-getitem?mr=#1}{#2}}
\providecommand{\href}[2]{#2}


\begin{thebibliography}{10}
\bibitem{atkinson:92} K.E.~Atkinson, 
	\emph{A survey of numerical methods for solving nonlinear integral equations}, 
	J.\ Integr.\ Equat.\ Appl.\ \textbf{4(1)} (1992), 15--46.


\bibitem{atkinson:han:05} K.E.~Atkinson, W.~Han, 
	\emph{Theoretical numerical analysis} ($2$nd edition), 
	Texts in Applied Mathematics 39, Springer, Heidelberg etc., 2000.


\bibitem{bortolan:carvalho:langa:20} M.~Bortolan, A.~Carvalho, J.~Langa, 
	{\em Attractors under autonomous and non-autonomous perturbations}, 
	Mathematical Surveys and Monographs \textbf{246} (2020), AMS, Providence, RI.

\bibitem{bouhours:lewis:16} J.~Bouhours, M.A.\ Lewis, 
	\emph{Climate change and integrodifference equations in a stochastic environment}, 
	Bull.\ Math.\ Biology \textbf{78(9)} (2016), 1866--1903.

\bibitem{carvalho:langa:robinson:10} A.N.\ Carvalho, J.A.\ Langa, J.C.\ Robinson, 
	\emph{Attractors for infinite-dimensional non-autonomous dynamical systems}, 
	Applied Mathematical Sciences 182, Springer, Berlin etc., 2012.

\bibitem{cheban:kloeden:schmalfuss:00} D.~Cheban, P.E.~Kloeden, B.~Schmalfu{\ss}, 
	\emph{Pullback attractors in dissipative nonautonomous differential equations under discretization}, 
	J.\ Dyn.\ Differ.\ Equations \textbf{13(1)} (2000), 185--213. 

\bibitem{cui:kloeden:yang:19} H.~Cui, P.E.~Kloeden, M.~Yang, 
\emph{Forward omega limit sets of nonautonomous dynamical systems}, 
Discrete Contin.\ Dyn.\ Syst.\ (Series B) \textbf{13} (2019), 1103--1114.

\bibitem{dubeau:savoie:96} F.\ Dubeau, J.\ Savoie, 
	\emph{Optimal error bounds for quadratic spline interpolation}, 
	J.~Math.\ Anal.\ Appl.~\textbf{198} (1996), 49--63. 

\bibitem{hackbusch:95} W.~Hackbusch, 
	\emph{Integral Equations -- Theory and Numerical Treatment}.
	Birkh{\"a}user, Basel etc., 1995.

\bibitem{hackbusch:14} \bysame, 
	\emph{The concept of stability in numerical mathematics}, 
	Series in Computational Mathematics 45, 
	Springer, Heidelberg etc., 2014.

\bibitem{hale:88} J.K.~Hale, 
	\emph{Asymptotic behavior of dissipative systems}, 
	Mathematical Surveys and Monographs~25, AMS, Providence, RI, 1988.

\bibitem{hale:lin:raugel:88} J.~Hale, X.-B. Lin, G.~Raugel, 
	\emph{Upper semicontinuity of attractors for approximations of semigroups and partial differential equations}, 
	Math.\ Comput.~\textbf{50(181)} (1988) 89--123.

\bibitem{hall:meyer:76} C.A.~Hall, W.W.~Meyer, 
	\emph{Optimal error bounds for cubic spline interpolation}, 
	J.~Approx.\ Theory~\textbf{16} (1976), 105--122. 

\bibitem{hammerlin:hoffmann:91} G.~H\"ammerlin, K.H.~Hoffmann, 
	{\em Numerical mathematics}, 
	Undergraduate Texts in Mathematics. Springer, New York etc., 1991.

\bibitem{hirsch:lacombe:1999} F.~Hirsch, G.~Lacombe,
	\emph{Elements of functional analysis},
	Graduate Texts in Mathematics 192, Springer, New York etc., 1999.

\bibitem{hoang:olson:robinson:15} L.T.~Hoang, E.J.~Olson, J.C.~Robinson, 
	\emph{On the continuity of global attractors}, 
	Proc.\ Am.\ Math.\ Soc.~\textbf{143(10)} (2015), 4389--4395. 

\bibitem{hoang:olson:robinson:18} \bysame, 
	\emph{Continuity of pullback and uniform attractors}, 
	J.\ Differ.\ Equations~\textbf{264(6)} (2018), 4067--4093. 

\bibitem{huynh:poetzsche:kloeden:20} H.~Huynh, C.~P\"otzsche, P.E.~Kloeden, 
	\emph{Forward and pullback dynamics of nonautonomous integrodifference equations: Basic constructions}, 
	J.\ Dyn.\ Diff.\ Equat.\ \textbf{33} (2021). 

\bibitem{jacobsen:jin:lewis:15} J.~Jacobsen, Y.~Jin, M.A.~Lewis, 
	\emph{Integrodifference models for persistence in temporally varying river environments}, 
	J.\ Math.\ Biol.\ \textbf{70} (2015), 549--590.

\bibitem{kammerer:etal:74} W.J.~Kammerer, G.W.~Reddien, R.S.~Varga. 
	\emph{Quadratic interpolatory splines}, 
	Numer.\ Math.~\textbf{22} (1974), 241--259. 

\bibitem{kloeden:lorenz:86} P.E.~Kloeden, J.~Lorenz, 
	\emph{Stable attracting sets in dynamical systems and in their one-step discretizations}, 
	SIAM J.\ Numer.\ Anal.\ \textbf{23(5)} (1986), 986--995. 

\bibitem{kloeden:lorenz:89} P.~Kloeden, J.~Lorenz, 
	\emph{Lyapunov stability and attractors under discretization}, 
	Differential Equations, Proceedings of the EQUADIFF Conference, C.M.\ Dafermos, G.\ Ladas, G.\ Papanicolaou (eds.), Marcel Dekker, Inc., New York, 361--368, 1989.

\bibitem{kloeden:lorenz:16} P.E.~Kloeden, T.~Lorenz, 
	\emph{Construction of nonautonomous forward attractors}, 
	Proc.\ Am.\ Math.\ Soc.\ \textbf{144(1)} (2016), 259--268.

\bibitem{kloeden:poetzsche:rasmussen:11} P.E.~Kloeden, C.~P{\"o}tzsche, M.~Rasmussen, 
	\emph{Limitations of pullback attractors for processes}, 
	J.\ Difference Equ.\ Appl.\ \textbf{18(4)} (2012), 693--701. 

\bibitem{kloeden:rasmussen:10} P.E.~Kloeden, M.~Rasmussen, 
	\emph{Nonautonomous dynamical systems}, 
	Mathematical Surveys and Monographs 176, AMS, Providence, RI, 2011.

\bibitem{kloeden:yang:16} P.E.~Kloeden, M.~Yang, 
	\emph{Forward attraction in nonautonomous difference equations}, 
	J.\ Difference Equ.\ Appl.\ \textbf{22(8)} (2016), 1027--1039.

\bibitem{kloeden:16} P.E.~Kloeden, 
	\emph{Asymptotic invariance and the discretization of nonautonomous forward attracting sets}, 
	J.\ Comput.\ Dynamics \textbf{3(2)} (2016), 179--189. 

\bibitem{kloeden:yang:21} P.E.\ Kloeden, M.\ Yang, 
	\emph{An introduction to nonautonomous dynamical systems and their attractors}, 
	World Sci.\ Publ., Hackensack, NJ, 2021. 

\bibitem{kot:schaeffer:86} M.~Kot, W.M.\ Schaffer, 
	\emph{Discrete-time growth-dispersal models}, 
	Math.\ Biosci.\ \textbf{80} (1986), 109--136.

\bibitem{lutscher:19} F.~Lutscher, 
	\emph{Integrodifference equations in spatial ecology}, 
	Interdisciplinary Applied Mathematics 49, Springer, Cham, 2019.

\bibitem{martin:76} R.H.~Martin, 
	\emph{Nonlinear operators and differential equations in Banach spaces}, 
	Pure and Applied Mathematics 11, John Wiley \& Sons, Chichester etc., 1976.

\bibitem{neuman:79} E.~Neuman, 
	\emph{Bounds for the norm of certain spline projections}, 
	J.\ Appprox.\ Theory \textbf{27} (1979), 135--145. 

\bibitem{poetzsche:10} C.~P{\"o}tzsche, 
	\emph{Geometric theory of discrete nonautonomous dynamical systems}, 
	Lect.\ Notes Math.\ 2002, Springer, Berlin etc., 2010.

\bibitem{poetzsche:15} \bysame, 
	\emph{Order-preserving nonautonomous discrete dynamics: Attractors and entire solutions}, 
	Positivity \textbf{19(3)} (2015), 547--576.

\bibitem{poetzsche:18} \bysame, 
	\emph{Numerical dynamics of integrodifference equations: Basics and discretization errors in a $C^0$-setting}, 
	Appl.\ Math.\ Comput.\ \textbf{354} (2019), 422--443. 
	

\bibitem{stuart:95} A.~Stuart, 
	Perturbation theory for infinite dimensional dynamical systems, 
	in M.~Ainsworth, J.~Levesley, W.~Light, M.~Marletta (eds.), {\em Theory and Numerics of Ordinary and Partial Differential Equations}, {\em Advances in Numerical Analysis} IV, pp.~181--290, Clarendon
 Press, Oxford, 1995.

\bibitem{stuart:humphries:98} A.~Stuart, A.~Humphries, 
	{\em Dynamical systems and numerical analysis}, 
	Monographs on Applied and Computational Mathematics $2$, University Press, Cambridge, 1998.

\end{thebibliography}
\end{document}